\documentclass[letterpaper,twoside]{article}
\usepackage{fancyhdr}
\usepackage[left=3.3cm, top=3.2cm, right=3.3cm, bottom=3.2cm, marginparwidth=2.6cm, marginparsep=5mm]{geometry}
\usepackage{tikz-cd}
\usepackage{amsmath} 
\usepackage{enumerate}
\usepackage{amsfonts}
\usepackage{amssymb}
\usepackage{mathrsfs}
\usepackage{faktor}
\usepackage{xcolor}
\usepackage{array}
\usepackage{blkarray}

\usepackage[disable]{todonotes}
\usepackage{xparse}
\usepackage{amsthm}
\usepackage{indentfirst}
\usepackage{float}
\usepackage{booktabs}
\usepackage{titlesec} 
\usepackage{ stmaryrd }
\usepackage{charter}
\usepackage{tgheros}
\usepackage{bbm}
\usepackage[T1]{fontenc}
\usepackage[expert]{mathdesign}
\usepackage[OT2,T1]{fontenc}
\DeclareSymbolFont{cyrletters}{OT2}{wncyr}{m}{n}
\DeclareMathSymbol{\Sha}{\mathalpha}{cyrletters}{"58}

\titleformat{\section}[hang]{
	\usefont{T1}{qhv}{b}{n}\selectfont} %
{} 
{0em}
{\hspace{-0.4pt}\Large \thesection\hspace{0.6em}}
\titleformat{name=\section,numberless}[display]
{
	\usefont{T1}{qhv}{b}{n}\selectfont} %
{} 
{0em}
{\Large}

\renewcommand{\refname}{References}

\makeatletter \g@addto@macro\@floatboxreset\centering 
\makeatother
\renewcommand{\sectionmark}[1]%
{\markright{{\thesection\ #1}}}
\renewcommand{\subsectionmark}[1]%
{}

\newcommand{\helv}{%
	\fontfamily{phv}\fontseries{b}\fontsize{9}{11}\selectfont}
\usepackage[ocgcolorlinks]{hyperref} %
\definecolor{linkcolour}{rgb}{0,0.2,0.6}
\hypersetup{colorlinks,breaklinks,urlcolor=linkcolour, linkcolor=linkcolour}
\DeclareDocumentCommand{\newfaktor}{s m O{0.5} m O{-0.5}}{%
	\setbox0=\hbox{\ensuremath{#2}}%
	\setbox1=\hbox{\ensuremath{\diagup}}%
	\setbox2=\hbox{\ensuremath{#4}}%
	\raisebox{#3\ht1}{\usebox0}%
	\mkern-5mu\ifthenelse{\equal{#1}{\BooleanTrue}}%
	{\diagup}%
	{\rotatebox{-44}{\rule[#5\ht2]{0.4pt}{-#5\ht2+#3\ht0+\ht0}}}%
	\mkern-4mu%
	\raisebox{#5\ht2}{\usebox2}%
}

\usepackage{xpatch}
\makeatletter
\xapptocmd{\@sect}{\csname #1mark\endcsname{#7}}{}{}
\makeatother

\newtheoremstyle{mystyle}%
{}%
{}%
{\itshape}%
{}%
{\fontsize{10.5pt}{0pt}\scshape }%
{}%
{.5em }%
{}%

\theoremstyle{mystyle}
\newtheorem{lemma}{Lemma}[section]
\newtheorem{theorem}[lemma]{Theorem}
\newtheorem{corollary}[lemma]{Corollary}

\newtheorem{proposition}[lemma]{Proposition}
\newtheorem*{claim}{Claim}

\newtheoremstyle{roman}%
{}%
{}%
{\normalfont}%
{}%
{\scshape}%
{}%
{.5em }%
{}%
\theoremstyle{roman}
\newtheorem{definition}[lemma]{Definition}
\newtheorem{example}[lemma]{Example}
\newtheorem{notat}[lemma]{Notation}
\newtheorem{constr}[lemma]{Construction}
\newtheorem{remark}[lemma]{Remark}

\newtheorem*{thm*}{Theorem}

\makeatletter
\renewenvironment{proof}[1][\proofname]{\par
	\pushQED{\qed}%
	\normalfont \topsep0\p@\@plus0\p@\relax
	\trivlist
	\item\relax
	{\itshape
		#1\@addpunct{.}}\hspace\labelsep\ignorespaces
}{%
	\popQED\endtrivlist\@endpefalse
}
\makeatother

\newcommand{\AbVar}{{\mathrm {AbVar}}}
\newcommand{\VHS}{{\mathrm {VHS}}}
\newcommand{\CHM}{{\mathrm {CHM}}}

\newcommand{\HomM}{{\mathrm {HomM}}}
\newcommand{\Rep}{\operatorname{Rep}}

\newcommand{\Et}{{\acute{\mathrm{E}} \mathrm{t}}}
\newcommand{\UVar}{\textrm{-}\mathrm{UVar}}
\newcommand{\Corr}{\operatorname{Corr}}
\newcommand{\SmProjVar}{\mathrm{SmProjVar}}

\newcommand{\diag}{\operatorname{diag}}

\newcounter{saveenumerate}
\makeatletter
\newcommand{\enumeratext}[1]{%
	\setcounter{saveenumerate}{\value{enum\romannumeral\the\@enumdepth}}
\end{enumerate}
#1
\begin{enumerate}[i)]
	\setcounter{enum\romannumeral\the\@enumdepth}{\value{saveenumerate}}%
}
\makeatother

\include{functions}

\newcommand{\Sh}{\operatorname{Sh}}

\renewcommand{\O}{\operatorname{O}}
\newcommand{\Sp}{\operatorname{Sp}}
\newcommand{\GSp}{\operatorname{GSp}}
\newcommand{\GU}{\operatorname{GU}}
\renewcommand{\U}{\operatorname{U}}

\newcommand{\MM}{\operatorname{M}}
\renewcommand{\M}{\operatorname{M}}
\newcommand{\GMT}{\operatorname{GMT}}

\edef\restoreparindent{\parindent=\the\parindent\relax}
\usepackage{parskip}
\restoreparindent

\newcommand{\Anc}{\operatorname{Anc}}

\renewcommand{\ab}{\mathrm{ab}}

\renewcommand{\opp}{\mathrm{op}}

\newcommand{\mot}{{ \textrm{\rm mot}}}
\newcommand{\AV}{{\textrm{\rm AV}}}
\renewcommand{\H}{\mathcal{H}}
\newcommand{\HH}{\mathbb{H}}
\begin{document}
	\title{\vspace{-2cm} Functoriality of motivic lifts of the canonical construction}
	\author{Alex Torzewski}
	\date{}
	\maketitle
	\begin{abstract}
		Let $(G,\mathfrak{X})$ be a Shimura datum and $K$ a neat open compact subgroup of $G(\A_f)$. Under mild hypothesis on $(G,\mathfrak{X})$, the canonical construction associates a variation of Hodge structure on $\Sh_K(G,\mathfrak{X})(\co)$ to a representation of $G$. It is conjectured that this should be of motivic origin. Specifically, there should be a lift of the canonical construction which takes values in relative Chow motives over $\Sh_K(G,\mathfrak{X})$ and is functorial in $(G,\mathfrak{X})$. Using the formalism of mixed Shimura varieties, we show that such a motivic lift exists on the full subcategory of representations of Hodge type $\{(-1,0),(0,-1)\}$. If $(G,\mathfrak{X})$ is equipped with a choice of PEL-datum, Ancona has defined a motivic lift for all representations of $G$. We show that this is independent of the choice of PEL-datum and give criteria for it to be compatible with base change. Additionally, we provide a classification of Shimura data of PEL-type and demonstrate that the canonical construction is applicable in this context.
	\end{abstract}
\label{relmotives}

\section{Introduction}

Let $(G,\mathfrak{X})$ be a Shimura datum. By design, there is a functor $\Rep(G)\to \VHS/\mathfrak{X}$ which assigns a $\q$-valued variation of Hodge structures on $\mathfrak{X}$ to a representation of $G$. For any neat open compact $K\le G(\A_f)$, let $S:=\Sh_K(G,\mathfrak{X})$ denote the corresponding Shimura variety, defined over its reflex field via canonical models. For well-behaved $(G,\mathfrak{X})$, the variations of Hodge structure constructed on $\mathfrak{X}$ descend to $S(\co)$. We call the resulting functor $\Rep(G) \to \VHS/S(\co)$ the \emph{canonical construction} and denote it by $\mu_G^\H$. 

The canonical construction should be of motivic origin. Specifically, there should be a canonical
$\ot$-functor $\mu_G^\mot \colon \Rep(G) \to \CHM/S$ to the category of relative Chow motives over $S$, such that 
	\[
	\begin{tikzcd}[column sep=tiny]
	\Rep(G)\ar{rr}{\mu_G^\mot} \ar{dr}[swap]{\mu_G^\H} && \CHM/S \ar{dl}{H^\bullet_B} \ar[dll, phantom, "\implies" rotate=-155, near start, start anchor={[xshift=-2ex]}, end anchor= north ] \\
	{}& \VHS/S(\co)
	\end{tikzcd}
\]
commutes up to canonical natural isomorphism. Here $H^\bullet_B$ denotes the relative Betti realisation enriched to take values in variations of Hodge structure. The functor $\mu_G^\mot$ should also be well behaved under change of $G$ etc. In particular, the canonical construction should produce variations of Hodge structure which arise from geometry.\todo{}

As an example, for the usual modular curve datum $(\GL_2, \mathcal{H})$, if $V$ denotes the standard representation of $\GL_2$, then $\mu_{G}^\H(V)$ is isomorphic to $H^1_B(\mathcal{E}\to S)^\vee$, where $\mathcal{E}\to S$ is the universal elliptic curve. The obvious choice for $\mu_{G}^\mot(V)$ is then the relative Chow motive $h^1(\mathcal{E}\to S)^\vee$ (in the notation of Theorem \ref{DenMur}).
 
Let $\Rep(G)^\AV$ denote the full subcategory of $\Rep(G)$ whose objects are of Hodge type $\{ (-1,0) ,\allowbreak (0,-1) \}$, i.e.\ for any $(h\colon \S \to G)\in \mathfrak{X}$, the restriction of $V$ to $\S$ is $(z\op \bar z)$-isotypical. Alternatively, the objects of $\Rep(G)^\AV$ are those for which their image under $\mu_G^\H$ is the dual of $H^1_B(A\to S)$ for some abelian variety $A\to S$ (see Lemma \ref{avforavreps}).

The first aim of this paper	is to show that $\mu_G^\mot$ can be defined on $\Rep(G)^\AV$ with the desired properties:
\begin{theorem}\label{intromain}
	Let $(G,\mathfrak{X})$ be a Shimura datum and $K\le G(\A_f)$ be a neat open compact subgroup. Write $S$ for the corresponding Shimura variety $\Sh_K(G,\mathfrak{X})$. There is a canonical $\ot$-functor $\mu_G^\mot\colon \Rep(G)^\AV\to \CHM/S$ for which the following diagram
\[		\begin{tikzcd}[column sep=tiny]
		\Rep(G)^\AV \ar{rr}{\mu_G^\mot} \ar{dr}[swap]{\mu_G^\H} && \CHM/S \ar{dl}{H^\bullet_B} \ar[dll, phantom, "\implies" rotate=-155, near start, start anchor={[xshift=-2ex]}, end anchor= north ] \\
		{}& \VHS/S(\co)
		\end{tikzcd}\]
	commutes up to a canonical natural isomorphism. Now let $f \colon (G',\mathfrak{X}')\to (G,\mathfrak{X})$ be a morphism of Shimura data and $K\le G(\A_f),K'\le G'(\A_f)$ neat open compact subgroups with $f(K')\le K$. Let $E'$ be the reflex field of $S'$, then we also denote by $f$ the induced map $S':=\Sh_{K'}(G,\mathfrak{X}') \to S_{E'}:=\Sh_K(G,\mathfrak{X})_{E'} \to S$ between the corresponding Shimura varieties. Then there is a commutative prism:
	\[	\begin{tikzcd}[column sep=tiny]
		\Rep(G)^\AV \ar{dd}[swap]{f^*} \ar{dr}[swap]{\mu_G^{\mathcal{H}}} \ar{rr}{\mu_G^\mot} && \CHM/S\ar{dd}{f^*} \ar{dl}{H^\bullet_B}\\
		& \VHS/S(\co)  & \\
		\Rep(G')^\AV \ar{dr}[swap]{\mu_{G'}^{\mathcal{H}}}\ar[rr,"\mu_{G'}^\mot" near start] && \CHM/S' \ar{dl}{H_B^\bullet}\\
		& \VHS/S'(\co) \arrow[ uu,crossing over, swap, "f^*" near end, leftarrow] & 
		\end{tikzcd}\]
	where the vertical maps are base change by $f$.
\end{theorem}
This is stated more precisely as Theorem \ref{tildenat}. Note that the reflex field of $(G',\mathfrak{X}')$ is allowed to be strictly larger than that of $(G,\mathfrak{X})$. The method of proof is to use the formalism provided by mixed Shimura varieties. Mixed Shimura varieties, as defined by Pink, generalise the traditional definition by allowing for non-reductive algebraic groups. Crucially, objects such as universal elliptic curves fit into this framework, i.e.\ they are mixed Shimura varieties and their structure maps are given by functoriality of mixed Shimura data.

Canonical constructions exist more generally than just the Hodge case. For example, the $\ell$-adic \'etale canonical construction associates a lisse $\ell$-adic sheaf on $S$ (considered as defined over its reflex field via canonical models) to a representation of $G$. The functor $\mu_G^\mot$ should lift every incarnation of the canonical construction. In Section \ref{etalecase}, we show this in the case of the \'etale canonical construction.

For PEL-type Shimura data much stronger results on lifting $\mu_G^\H$ are known due to work of Ancona \cite{Anconapaper}. For Shimura data with a fixed choice of PEL-datum, Ancona has been able to define a functor $\Anc_{G}$ defined on all of $\Rep(G)$ (see Thm.\ \ref{ancfunc}). Unfortunately, it is not directly clear that $\Anc_G$ commutes with pull back via a morphism of Shimura data. Moreover, it is not clear that $\Anc_G$ is independent of the choice of PEL-datum (recall that a Shimura variety may admit multiple distinct PEL-data, see Example \ref{twoPELdata}).

In the latter part of this paper, we show that $\Anc_G$ is independent of the choice of PEL-datum (Lemma \ref{identityadmissible} and Theorem \ref{main}) and in many cases commutes with morphisms of Shimura varieties. More precisely, call a morphism of Shimura data each with chosen (possibly unrelated) PEL-data $f\colon (G',\mathfrak{X}')\to (G,\mathfrak{X})$ \emph{admissible} if
\[ f^*V\textrm{ is a summand of $V'^{\op k}$ for some $k$},    \]
as $G'$-representations, where $V',V$ denote the representations given in the PEL-data on the source and target respectively. The motivation for this definition is that it ensures that we may use functoriality of mixed Shimura data to compare $f^*\Anc_G(V)$ and $\Anc_{G'}(V')$.
\begin{theorem}
Given $f\colon (G',\mathfrak{X}')\to (G,\mathfrak{X})$ an admissible morphism of PEL-type Shimura varieties each with chosen PEL-data, then the following diagram commutes:\todo{}
\[
	\begin{tikzcd}
	\Rep(G) \ar{d}[swap]{f^*} \ar{r}{\Anc_G}\arrow[dr,phantom,  "\implies" rotate=-145]  & \CHM/S \ar{d}{f^*} \\
	\Rep(G') \ar{r}[swap]{\Anc_{G'}} & \CHM/S'
	\end{tikzcd}
\]
up to a specified natural isomorphism. Moreover, there is a prism analogous to that of Theorem \ref{intromain}.
\end{theorem}

This is made precise in Theorem \ref{main}.

Not all morphisms $f$ are admissible (see Example \ref{admissibilitycounterexample}), but in Corollary \ref{symp/orth} we show that every $f$ for which the source only has factors of symplectic type (see Lemma \ref{PELsatisfySV5}) is admissible. In any case, it is easy to decide if a given morphism is admissible.

One application of results such as the above is in the theory of Euler systems. In this context it is often required to pullback classes lying in the cohomology of Shimura varieties under morphisms of Shimura data. It is also necessary to switch between various cohomology theories. For this reason it is desirable to be able to perform such operations at the motivic level. There has been significant recent progress in this direction due to Lemma's construction of motivic classes on Siegel threefolds \cite{Lemma}. If functoriality results are available Lemma's classes have the potential to yield Euler systems for a multitude of different Shimura varieties (see for example \cite{LSZ}, particularly Section 6).

One other observation from practical applications is that it is desirable to have such results with $F$-coefficients for $F/k$ a number field. For this reason all the following is phrased to allow for coefficients.

Finally, in Section \ref{sec:class} we provide a self-contained classification of the groups arising from PEL-data (see Lemma \ref{PELclass}), which as a consequence demonstrates that PEL-type Shimura data are sufficiently well-behaved to apply the canonical construction. This is something which is well-known, but for which we are not aware of a reference for.

{\bfseries Acknowledgements:} I would like to especially thank David Loeffler for suggesting the research topic and for providing guidance throughout. I am also deeply indebted to Giuseppe Ancona for many helpful discussions and explaining his results to me. I would also like to thank an anonymous referee for their helpful feedback and suggestions.
\section{Relative motives}\label{relmot}
We now recall some background on relative motives.
\begin{notat}
	Assume that $k$ is a field of characteristic zero equipped with a fixed embedding into $\co$. Given a $k$-variety $Z$, we write $Z(\co)$ for its complex points considered as a complex manifold.
	
	In this section, we fix $S$ to be a smooth quasi-projective $k$-scheme. For simplicity, we shall assume that all components of $S$ have the same dimension $d_S$.
\end{notat}
\begin{definition}\label{d:relmot} 
	Following \cite[Sec.\ 1]{DeningerMurre}, fix an adequate equivalence relation $\sim$ on all $k$-varieties and let $X,Y$ be smooth projective $S$-schemes. Assume for simplicity that $X,Y$ are equidimensional of dimensions $d_X,d_Y$ respectively. We define the group of degree $p$ correspondences from $X$ to $Y$, up to equivalence by $\sim$, to be
	\[  \Corr^p_S(X,Y)=A_\sim^{d_X-d_S+p}(X\ti_S Y),  \]
	where $A^d_\sim(-)$ denotes the $\q$-vector space of codimension $d$ cycles up to equivalence by $\sim$. Proceeding as in the classical case we obtain the category $\mathcal{M}_\sim/S$ of \emph{relative motives over $S$ with respect to $\sim$}, whose objects are triples $(X,e,n)$ consisting of a variety $X$, an idempotent $e \in \Corr^0_S(X,X)$ and an integer $n\in \z$ corresponding to Tate twists. The category $\mathcal{M}_\sim/S$ is a $\q$-linear $\ot$-category, with the tensor structure being given by fibre product over $S$.
	
	We are mostly concerned with the case when $\sim$ is taken to be rational equivalence $\sim_\textrm{rat}$, in which case we denote $\mathcal{M}_\sim/S$ by $\CHM/S$, or homological equivalence $\sim_\textrm{hom}$ with respect to singular cohomology (or equivalently any choice of $\ell$-adic cohomology), in which case we denote the resulting category by $\HomM/S$. These categories are referred to as \emph{relative Chow motives over $S$} and \emph{relative homological motives over $S$} respectively\footnote{It may be better to refer to $\HomM/S$ as ``naive homological motives''. This is because, unlike in the case of $S=k$, our homological motives admit non-trivial maps between objects which should be considered to live in different cohomological degrees. As a result, they do not coincide with what we may reasonably expect of ``relative numerical motives''.}. Write $H^i_B(Z(\co),\q)$ for the singular cohomology of a variety $Z/k$. Since homological equivalence is coarser than rational equivalence
	we obtain a forgetful map
	\[ \CHM/S \to \HomM/S,  \]
	which is full.
	
	If $\SmProjVar/S$ denotes the category of (not necessarily irreducible) smooth projective varieties over $S$, then there is a functor $h\colon (\SmProjVar/S)^\opp\to \CHM/S$ which assigns to a variety $X/S$ its motive $(X,\Delta_X, 0)$ where $\Delta_X$ is the diagonal cycle of $X\ti_S X$. The same is also true of homological motives.
\end{definition}
For any adequate equivalence relation, the construction of $\mathcal{M}_\sim/S$ is compatible with change of $S$, i.e.\ given $f\colon S'\to S$, we obtain pullback functors $f^*\colon \mathcal{M}_\sim/S \to \mathcal{M}_\sim/S'$ by base changing triples in the obvious way.
\begin{remark}This construction has been extended to the case when $S\to k$ is quasi-projective but not necessarily smooth by Corti--Hanamura \cite{CortiHanamura}.
\end{remark}

\begin{definition}\label{Fvaluedmotives}
	Let $F/\q$ be a number field. We define $(\CHM/S)_F$ to be the category with the same objects as $\CHM/S$ but for which $\Hom_{(\CHM/S)_F/S}(A,B)=\Hom_{\CHM/S}(A,B)\ot_\q F$. We then define $\CHM_F/S$ to be the pseudo-abelianisation $((\CHM/S)_F)^\natural$ of $(\CHM/S)_F$ and refer to it as the category of \emph{relative Chow motives over $S$ with coefficients in $F$}. We shall frequently use that it is equivalent to think of a Chow motive with coefficients in $F$ as an object $M$ of $((\CHM/S)_F)^\natural$ or as an object of $\CHM/S$ together with an inclusion $F\hookrightarrow \End_{\CHM/S}(M)$ (see \cite[Sec.\ 2]{Deligne} or for more details \cite[Sec.\ 5]{AndreKahn}). We define $\HomM_F/S$ analogously.
\end{definition}
\begin{definition}\label{abmot}
	Let $\AbVar/S$ denote the category of abelian varieties over $S$. We denote by $\CHM^\ab_F/S, \HomM^\ab_F/S$ the smallest rigid linear symmetric tensor subcategories which contain the motives of abelian varieties and are closed under taking subobjects and Tate twists.
\end{definition}
\begin{theorem}\label{O'Sullivan}There is a unique section $\mathcal{I}$ of the projection $\mathcal{N}\colon \CHM_F^\ab/S \to \HomM_F^\ab/S$ which is a linear symmetric tensor functor, commutes with Tate twists and is such that
	\[  h |_{\textrm{\rm AbVar}^\opp} = \mathcal{I}\circ \mathcal{N} \circ h|_{\rm \textrm{\rm AbVar}^\opp}.   \]
\end{theorem}
\begin{proof}
	This follows from work of O'Sullivan \cite[pf.\ of Thm.\ 6.1.1]{O'Sullivan} (see also \cite[Thm.\ 7.1]{Anconapaper}). More precisely, O'Sullivan checks that any quotient of ``Chow theory'' by a proper ideal has a right inverse which is unique subject to the above conditions. But cycles which are homologically equivalent to zero form a proper ideal within $\CHM/S$. The same reasoning applies for motives with coefficients.
\end{proof}
\begin{remark}\label{O'Sullivanbc}
	Morphisms in the image of $\mathcal{I}$ are \emph{symmetrically distinguished} in the sense of \cite[Def.\ 6.2.1]{O'Sullivan}. O'Sullivan checks that the pullback of a symmetrically distinguished cycle is symmetrically distinguished \cite[Thm.\ iii) p2]{O'Sullivan}. From this, it is easy to see check that, given a morphism $f\colon S' \to S$, there is a natural isomorphism $f^*\circ\mathcal{I}\implies \mathcal{I}\circ f^*$ since both compositions yield a symmetrically distinguished Chow cycle lying over a numerical cycle, but there is only one such cycle (cf.\ \cite[Thm.\ 6.2.5]{O'Sullivan}).
\end{remark}

\begin{theorem}[{\cite[Thm.\ 3.1]{DeningerMurre}}]\label{denmurre}
	Let $A/S$ be an abelian variety of dimension~$n$, then within $\CHM_F/S$ there is a decomposition
	\[ h(A) = \bigoplus_{i=0}^{2n} h^i(A), \]
	such that, if $[n]\colon A \to A$ denotes multiplication by $n$, then
	\[ h([n])= \bigoplus_{i=0}^{2n} n^i\cdot \id_{h^i(A)} .  \]\label{DenMur}
\end{theorem}%
The analogous statement for homological motives also holds but is automatic. The second condition ensures that the decomposition is compatible with change of $A$ and $S$ as well as applying any of the standard realisations. Another consequence is:%
\begin{theorem}[K\"unneth Formula]\label{kunneth} The decomposition $h(A)=\bigoplus_i h^i(A)$ respects the K\"unneth formula, i.e.\
	\[ h^k(A\ti A')=  \bigoplus _{i+j=k} h^i(A)\ot h^j(A') . \]
\end{theorem}
\begin{theorem}[{\cite[Prop.\ 2.2.1]{Kings}}]\label{kings}
	Given an abelian variety $A/S$, the map 
	\[\End(A)^\opp \ot F \to \End_{\CHM_F/S}(h^1(A))\]
	is an isomorphism.
\end{theorem}
\section{Realisations}
\label{sec:realisations}
\begin{notat}
 Let $S \overset{t}{\to}k$ be a smooth quasi-projective variety over a number field and $\VHS/S(\co)$ denote the category of $\q$-valued variations of Hodge structure on $S(\co)$. For any finite field extension $F/\q$, we may define $\VHS_F/S(\co)$ analogously to Definition \ref{Fvaluedmotives} (note we do not require $F\subset \re$).
\end{notat}
\begin{lemma}For any $S/k$ a smooth quasi-projective variety. There are relative Hodge realisation functors
	\[ H_B^\bullet \colon \HomM_F/S \to \VHS_F/S(\co),  \]
	which send $h(X\overset{p}{\to}S)(i)$ to $\bigoplus_j R^{j}p_*F_{X(\co)}(i)$. These are natural in $S$. \label{Hodge realisation}
\end{lemma}
This construction is spelt out in \cite[Cor.\ 4.5.7]{thesis}.
\begin{remark}\label{faithful}
	In contrast to the case when $S$ is a field, the relative Hodge realisation functors are not faithful in general. This is due to the presence of non-trivial morphisms between objects which are pure of different weights. In their work, Corti--Hanamura correct this by introducing a realisation functor taking values in a derived category. This is not necessary for our purposes as we shall only require faithfullness for elements of $\Hom_{\HomM_F/S}(h^i(X),h^i(Y))$ with $X,Y$ abelian varieties, which is true of $H^\bullet_B$ (cf.\ \cite[Remark 4.5.8]{thesis}). Note that for abelian varieties $H^i_B(X(\co))=H^\bullet_B(h^i(X))$, by Theorem \ref{DenMur}.
\end{remark}
\begin{remark}\label{xi} In Lemma \ref{Hodge realisation}, by naturality in $S$ we mean that given $f\colon S' \to S$, there is a natural isomorphism $\xi \colon f^*\circ H^\bullet_B \implies H^\bullet_B \circ f^*$. For an object $X\overset{p}{\to}S$ this is given by the proper base change map $f^*R^i p_*F_{X(\co)}\to R^i p_{S',*} f^*F_{ X(\co)}$.
\end{remark}

All the above also holds in the \'etale case, which we now record for use in Section \ref{etalecase}.
\begin{notat}
	Let $\ell$ be any prime and $\lambda$ a prime of $F$ dividing $\ell$. Given a scheme $X$, we write $\Et_{\lambda}/S$ for the category of lisse $F_\lambda$-sheaves on $X$ and $F_{\lambda,X}$ for the constant $F_\lambda$-sheaf on a scheme $X$ with coefficient group $F_\lambda$.\end{notat}
\begin{lemma}\label{etalerealisation}
	There are relative \'etale realisation functors
	\[ H^\bullet_\lambda \colon \HomM_F/S \to \Et_{\lambda}/S,\]
	which send $X\overset{p}{\to} S$ to $\bigoplus_i R^ip_*F_{\lambda,X}$. These are natural in $S$.
\end{lemma}

\section{The canonical construction}
\label{functors}
\begin{notat}
	For an algebraic group $G/\q$ and a field $F$ of characteristic zero let $\Rep_F(G)$ denote the category of representations of $G_F$ over $F$. We shall usually consider an object $V\in \Rep_F(G)$ as a representation $V$ of $G$ over $\q$ together with a map $F\hookrightarrow \End_G(V)$. We also set $\Rep(G):=\Rep_\q(G)$. 
\end{notat}
\begin{notat}
	Throughout $(G,\mathfrak{X})$ will denote a Shimura datum (which we often interchange with $(G,h)$ for $h\in \mathfrak{X}$). We shall always assume that our Shimura data are such that the identity connected component of the centre of $G$ is an almost-direct product of a $\q$-split torus and an $\re$-anisotropic torus. This ensures that all real cocharacters of the centre are in fact defined over $\q$.

	Upon fixing a choice of neat open compact $K\le G(\A_f)$, we denote $\Sh_K(G,\mathfrak{X})$ by $S$, always considered to be defined over the reflex field. We follow a similar convention for $(G',\mathfrak{X}')$ with $K'\le G'(\A_f)$ etc. If $f\colon (G',\mathfrak{X}') \to (G,\mathfrak{X})$ is a morphism of Shimura data for which $f(K')\le K$, then we also denote by $f$ the induced map $f \colon S'\to S$, even when the reflex fields may decrease.
\end{notat}
\begin{constr}[{\cite[Ch.\ 1]{Pink}}]\label{cancon}
	Given an element $(\rho \colon G_F \to \GL(V))\linebreak\in \Rep_F(G)$, we may define a variation of Hodge structure on $S(\co)$ as follows: consider $V$ as $\q$-representation of $G$ together with an action of $F$. Then the underlying local system corresponds to the cover 
	\[ G(\q)\ba ( \mathfrak{X} \ti (G(\A_f)/K) \ti V)\to G(\q)\ba (\mathfrak{X} \ti (G(\A_f)/K)),  \]
	where $g\in G(\q)$ acts by $(h_x,t,v) \mapsto (gh_xg^{-1}, gt,\rho(g)v)$. The stalk at a point $(h_x,t)$ is identified with corresponding fibre $\{ (h_x,t ,v ) \mid v \in V \} \cong V$ and as such may be given the $\q$-Hodge structure defined by the map $\rho\circ h_x \colon \S \to G_\re \to \GL(V_\re)$. This is independent of the choice of representative and can be checked to define a variation of Hodge structure (this uses the almost-direct product condition on the centre of $G$, for more details see \cite[Ch.\ 1]{Pink}).%
	This extends to a functor $\mu_G^\mathcal{H} \colon \Rep_F(G) \to \VHS_F/S(\co)$ referred to as the \emph{canonical construction} (where $\mathcal{H}$ stands for Hodge).
\end{constr}

\begin{constr}\label{canconbc} Let $V\in \Rep_F(G)$ and $f$ be as above. There is a canonical isomorphism of local systems $\kappa_V \colon f^*\mu_G^{\mathcal{H}}(V)\to \mu_{G'}^{\mathcal{H}}(f^*V)$ and this is also a morphism of variations of Hodge structure as it respects the Hodge structure on each fibre. The collection $\kappa := (\kappa_V)_V$ then defines a natural isomorphism:
\[
		\begin{tikzcd}
			\Rep_F(G) \ar{d}[swap]{f^*} \ar{r}{\mu_G^{\mathcal{H}}}\arrow[dr,phantom,  "\implies" rotate=-145]  & \VHS_F/S(\co) \ar{d}{f^*} \\
			\Rep_F(G') \ar{r}[swap]{\mu_{G'}^{\mathcal{H}}} & \VHS_F/S'(\co)
		\end{tikzcd}
\]
\end{constr}
\section{Mixed Shimura varieties}\label{msv}
Mixed Shimura data, as defined by Pink \cite{Pink}, extend the traditional definition to not necessarily reductive algebraic groups. We briefly recall the basic properties of mixed Shimura data, but in the restricted setting of where the homogeneous space is a conjugacy class of morphisms from the Deligne torus (as opposed to a finite cover of such a space) as this is true of all the data we shall consider.

A \emph{mixed Shimura datum} consists of a pair $(P,\tilde{\mathfrak{X}})$ with $P/\q$ a connected algebraic group and a subspace $\tilde{ \mathfrak{X}}\subseteq \Hom(\S_\co,P_\co)$ satisfying the requirements of \cite[Sec.\ 2.1]{Pink}. In the case that $P$ is reductive, i.e.\ that $P$ has trivial unipotent radical, we recover the classical definition of Shimura data, which we shall refer to as the \emph{pure} case.

For any neat open compact $K\le P(\A_f)$, there is an associated mixed Shimura variety $\Sh_K(P,\tilde{ \mathfrak{X}})$, which is algebraic over its reflex field \cite[Thm.\ 11.18]{Pink}. A morphism of mixed Shimura data $f\colon (P',\tilde{ \mathfrak{X}}')\to (P,\tilde{ \mathfrak{X}})$ is a map $P'\to P$ for which $f(\tilde{ \mathfrak{X}}')\subseteq \tilde{ \mathfrak{X}}$. Pairs of neat open compact subgroups $K\le P(\A_f)$ and $K' \le P'(\A_f)$ with $f(K')\le K$ give rise to algebraic maps $\Sh_{K'}(P',\tilde{ \mathfrak{X}}')\to \Sh_K(P,\tilde{ \mathfrak{X}})$ \cite[Sec.\ 3.4]{Pink}. 

Any mixed Shimura datum $(P,\tilde{ \mathfrak{X}})$ admits a map to the pure Shimura datum $(G, \mathfrak{X})$ where $G$ is the quotient of $P$ by its unipotent radical and $\mathfrak{X}$ is given by postcomposing elements of $ \tilde{ \mathfrak{X}}$ with $\pi\colon P \to G$ (cf.\ \cite[Prop.\ 2.9]{Pink}).

We shall always assume that our mixed Shimura varieties satisfy the stronger condition that: the centre of $G=P/R_u(P)$ is an almost-direct product of a $\q$-split torus and a torus which is $\re$-anisotropic (so the weight cocharacter $\pi \circ h \circ w\colon \G_{m,\re}\to G_\re$ is rational for $h\in \tilde{ \mathfrak{X}}$). These ensure that there is a canonical construction for mixed Shimura varieties associating variations of mixed Hodge structure on $\Sh_K(P,\tilde{ \mathfrak{X}})$ to representations of $\Rep(P)$ (see \cite[Sec.\ 1.18]{Pink}).

Universal abelian varieties can be seen as instances of mixed Shimura varieties (see Example \ref{univabmixedshim}). In this section, we shall observe that the theory of mixed Shimura varieties automates the creation of certain abelian varieties over pure Shimura varieties in a functorial way.

\begin{definition}
	Let $(G,\mathfrak{X})$ be a (pure) Shimura datum and $V\in \Rep_F(G)$. We consider $V$ as a $\q$-representation together with an $F$-structure $F \hookrightarrow \End_G(V)$. For any choice of $h_x \in \mathfrak{X}$, $V\ot_{\q} \co$ decomposes as a direct sum of one dimensional $\co$-subrepesentations upon each of which $z \in \S(\re)=\co^\ti$ acts as multiplication by $z^{-p_i}\bar z^{-q_i}$ for some $p_i,q_i$. We say that $V$ has \emph{Hodge type} given by set $\{(p_1,q_1),(p_2,q_2),...,(p_n,q_n)\}$ of $(p_i,q_i)$ occurring in the above decomposition. Since different choices of $h_x$ define isomorphic $\re$-Hodge structures, this is independent of the choice of $h_x$. 

	The Hodge type of a representation $V\in \Rep(G)$ coincides with the Hodge type of $\mu_G^\mathcal{H}(V)$ as a variation of Hodge structure on $S(\co)$.
\end{definition}

\begin{notat}\label{crapnotation}
	Let $\Rep_F(G)^\AV$ denote the full subcategory of $\Rep_F(G)$ whose objects have Hodge type contained in $\{(-1,0),(0,-1)\}$.
	
	Given $V\in \Rep_F(G)^\AV$, considering $V$ as a representation over $\q$, we may form the semi-direct product $V\rti G$ as an algebraic group over $\q$. Let $p \colon V\rti G\to G$ denote the projection map and $\tilde{ \mathfrak{X}}$ consist of the elements $t \in \Hom(\S_\co, (V\rti G)_\co)$ for which $p \circ t \in \mathfrak{X}_\co$.
\end{notat}
\begin{lemma}\label{VHS_F/SShimuradatum}
	Let $(G,\mathfrak{X})$ be a (pure) Shimura datum $V\in \Rep_F(G)^\AV$, then $(V\rti G,\tilde{\mathfrak{X}})$ is a mixed Shimura datum. 
\end{lemma}
\begin{proof}
	The unipotent radical of $V\rti G$ is $V$. If, in the notation of \cite[Sec.\ 2.1]{Pink}, we set $U=V$, then it is easy to check the conditions directly. Alternatively, use that $(V\rti G,\tilde{\mathfrak{X}})$ is an instance of a unipotent extension in the sense of \cite[Prop.\ 2.17]{Pink}. Note that we are assuming $(G,\mathfrak{X})$ has rational weight and the centre is an almost-direct product of a $\q$-split and $\re$-anisotropic torus. The datum $(V\rti G,\tilde{\mathfrak{X}})$ then also satisfies the corresponding strengthened condition of a mixed Shimura variety.
\end{proof}
Mixed Shimura data of the form $(V\rti G, \tilde{\mathfrak{X}})$ are the only non-pure data we shall need to consider. 
\begin{lemma}
	Let $(G,\mathfrak{X})$ be a Shimura datum,  $K\le G(\A_f)$ a neat open compact subgroup and $V\in \Rep(G)$. Then for any choice of $K$-stable full rank $\hat \z$-lattice $L\le V(\A_f)$, $L\rti K$ is a neat open compact subgroup of $V\rti G$.
\end{lemma}
\begin{proof}
This is an easy exercise, for example see \cite[pf.\ of Lem.\ 4.7.4]{thesis}.
\end{proof}
\begin{lemma}\label{avforavreps}
For any Shimura datum $(G,\mathfrak{X})$ and $V,K,L$ as above, the map \linebreak $ \Sh_{L\rti K}(V\rti G,\tilde{ \mathfrak{X}}) \to \Sh_K(G,\mathfrak{X})$ has the structure of an abelian variety.
\end{lemma}
\begin{proof}
	This is \cite[3.22 a)]{Pink} (the zero section is given by the Levi section $\iota \colon G \to V\rti G$). %
\end{proof}
Moreover, this is functorial in the sense that given a homomorphism of representations $f \colon V \to V'$ with $L \le V(\A_f)$ and $L'\le V'(\A_f)$ and $f(L) \le L'$, then the induced map of mixed Shimura varieties $\Sh_{L\rti K}(V\rti G , \tilde{ \mathfrak{X}})\to \Sh_{L\rti K}(V\rti G,\tilde{ \mathfrak{X}}')$ respects the group structure. The existence of the projection and identity section maps force $(V\rti G,\tilde{\mathfrak{X}})$ to have the same reflex field as $({G},\mathfrak{X})$ \cite[Sec.\ 11.2(b)]{Pink}.
\begin{example}\label{univabmixedshim}
	If a Shimura datum $(G,\mathfrak{X})$ has a PEL-datum with standard representation $V$ (see Definition \ref{PEL}), then for any neat open compact $K$ and $K$-stable $\hat \z$-lattice $L$ of $V(\A_f)$ (we shall always take our lattices to be of full rank),  $\Sh_{L\rti K}(V\rti~G,\tilde{ \mathfrak{X}})\to \Sh_K(G,{ \mathfrak{X}})$ is isogeneous to the universal abelian variety defined by the PEL-datum.
\end{example}
\begin{lemma}\label{fibreprodprop}
	\begin{enumerate}[i)]
		\item Let $(G,h)$ be a Shimura datum and $K$ a neat open compact subgroup. Given $V,W\in \Rep_F(G)^\AV$ and $K$-stable $\hat\z$-lattices $L_V,L_W$ of $V(\A_f),$ $W(\A_f)$, then as abelian varieties over $S$, there is a canonical isomorphism
		\begin{align*} \Sh_{(L_V\oplus L_W )\rti K}((V\oplus W)\rti G, &\tilde{ \mathfrak{X}}_{V\oplus W})\\
		&\cong \Sh_{L_V\rti K}( V\rti G, \tilde{ \mathfrak{X}}_{V})\ti_S  \Sh_{L_W\rti K}( W\rti G, \tilde{ \mathfrak{X}}_{W}),\end{align*}
		where $\tilde{ \mathfrak{X}}_{V},\tilde{ \mathfrak{X}}_{V},\tilde{ \mathfrak{X}}_{V\oplus W}$ are as in Notation \ref{crapnotation}.
		\item Given a morphism of pure Shimura data $f\colon (G',h')\to (G,h)$, neat open compact subgroups $K'\le G'(\A_f), K\le G(\A_f)$ with $f(K')\le K$, and $V\in \Rep_F(G)^\AV$ together with a $K$-stable $\hat \z$-lattice $L$, then there is a canonical isomorphism of abelian $S'$-schemes
		\[ \Sh_{L\rti K}(V\rti G, \tilde{ \mathfrak{X}})_{\Sh_{K'}(G',{ \mathfrak{X}'})} \cong  \Sh_{f^*L\rti K'}(f^*V\rti G', \tilde{ \mathfrak{X}'}),
		\]
		where $f^*L$ is the lattice $L$ considered as a $K'$-stable $\hat{\z}$-lattice.
	\end{enumerate}
\end{lemma}
\begin{proof}
	Both statements follow immediately from the characterisation of fibre products for mixed Shimura data given in \cite[Sec.\ 3.10]{Pink}.
\end{proof}
\begin{constr}
	We now define a functor $\mu_G^\mot\colon \Rep_F(G)^\AV\to \CHM_F/S$ as follows: given $V\in \Rep_F(G)^\AV$, let $L$  be a full rank $\hat \z$-lattice of $V(\A_f)$ which is stable under $K$. We then set $\mu_G^\mot(V)=h^1(\Sh_{L\rti K}(V\rti K,\tilde{\mathfrak{X}}))^\vee$ as a motive with rational coefficients which we equip with an $F$-structure $F\hookrightarrow \End_{\CHM/S}(\mu_G^\mot(V)) $ in the following way. Let
	\[  T:=\{ \alpha \in \End_G(V)\mid \alpha(L)\subseteq L\}.  \]
	For any $\alpha\in \End_G(V)$, $\alpha(L)$ is a $\hat \z$-lattice and so there exists an $n\in \n$ such that ${n\cdot \alpha(L)\le L}$. In other words, $T\ot_\z \q = \End_G(V)$. Thus, we may act on $h^1(\Sh_{L\rti K}(V\rti K,\tilde{\mathfrak{X}}))^\vee$ by $F=(T\cap F) \ot_\z \q$ with the first factor acting via functoriality of mixed Shimura varieties and the second by $\q$-linearity of $\CHM/S$. This uses that the actions of $\z$ as subring of $T\cap F$ (i.e.\ addition via the group law as an abelian variety) and as a subring of $\q$ coincide, which follows from Theorem \ref{DenMur}.  In contrast, this would not be true for $h(S_{K,V})^\vee$ and this does not define an element of $\CHM_F/S$.

	Given a morphism $f \colon V\to V'$, let $n\in \n$ be large enough to ensure that $f(nL)\le L'$. We obtain maps
	\[ (\pi_n^\vee)_* \colon h^1(\Sh_{L\rti K}(V\rti G, \tilde{\mathfrak{X}}))^\vee   \to h^1(\Sh_{nL\rti K}(V\rti G, \tilde{\mathfrak{X}}))^\vee,   \]
	\[   f_* \colon  h^1(\Sh_{nL\rti K}(V\rti G, \tilde{\mathfrak{X}}))^\vee \to  h^1(\Sh_{L'\rti K}(V'\rti G, \tilde{\mathfrak{X}}'))^\vee , \]
	where the first map is obtained by applying $h^1(-)^\vee$ to the dual of the map of abelian varieties $\pi_n \colon \Sh_{nL\rti K}(V\rti G,\tilde{\mathfrak{X}})\to \Sh_{L\rti K }(V\rti G,\tilde{\mathfrak{X}})$ which is given by functoriality of mixed Shimura varieties, whilst the second is $h^1(-)^\vee$ of the map of mixed Shimura varieties induced by $f$. We then set $\mu_G^\mot(f)$ to be $1/n$ times the composite $f_* \circ (\pi_m^\vee)_*$. By construction the morphisms $\mu_G^\mot(f)$ will respect the $F$-action.
\end{constr}
\begin{proposition}\label{pain!}
	Given a choice of $\hat \z$-lattice for each $V\in \Rep_F(G)^\AV$ as above, then the corresponding $\mu_G^\mot$ is a well-defined $\ot$-functor $\Rep_F(G)^\AV\to \CHM_F/S$. The functor $\mu_G^\mot$ is independent of the choice of lattice for each $V$, up to canonical natural isomorphism.
\end{proposition}
\begin{proof}
	We first remark that $\mu_G^\mot(f)$ is independent of the choice of $n$. This follows as the constructions for $n$ and for $nm$ differ by $1/m\cdot (\pi_m^\vee)_*\circ \pi_{m,*}=1/m\cdot [m]_*$, but, for an abelian variety $A/S$, $[m]$ acts on $h^1(A)^\vee $ by multiplication by $m$ (Theorem \ref{denmurre}).%
	
	That $\mu_G^\mot$ respects composition follows from the commutativity of the following diagram, for any $f \colon V\to V'$,  $m\in \n$ and $n$ s.t.\ $f(nL)\le L'$
\[
		\begin{tikzcd}
			\Sh_{nL\rti K}(V\rti G,\tilde{ \mathfrak{X}})\ar{r}{\pi_m^\vee}\ar{d}[swap]{f}& \Sh_{mnL\rti K}(V\rti K,\tilde{ \mathfrak{X}})\ar{d}{f}\\
			\Sh_{L'\rti K}(V'\rti G,\tilde{ \mathfrak{X}}')\ar{r}[swap]{{\pi_m^\vee}} &\Sh_{mL'\rti K}(V'\rti G,\tilde{ \mathfrak{X}}') 
		\end{tikzcd}
\]
	and thus it is clear that $\mu_G^\mot$ defines a functor.
	
	Given choices $L_1,L_2$ for each $V$ and corresponding functors $\mu_{G,1}^\mot,\mu_{G,2}^\mot$, we define a natural transformation $\psi \colon \mu_{G,1}^\mot \to \mu_{G,2}^\mot$ by setting $\psi_V$ to be $1/n$ times the map 
	\[h^1(\Sh_{L_1\rti K}(V\rti G,\tilde{ \mathfrak{X}}))^\vee \overset{(\pi_n^\vee)_*}{\to} h^1(\Sh_{nL_1\rti K}(V\rti G,\tilde{ \mathfrak{X}}))^\vee \overset{\id_*}{\to} h^1(\Sh_{L_2\rti K}(V\rti G,\tilde{ \mathfrak{X}}))^\vee\] for any $n$ such that $nL_1 \le L_2$. That this defines a natural transformation again follows from the commutativity of the above square. Moreover, as isogenies become invertible after applying $h^1(-)^\vee$, we find that $\psi$ defines a natural isomorphism.
\end{proof}
\begin{remark}
	If $f \colon V\to W$ is a non-zero homomorphism of representations of $G$ over $\q$ and we fix a neat open compact subgroup $K$ of $G$ and $K$-stable $\hat \z$-lattices $L_V\le V$, $L_W\le W$ such that $f(L_V)\le L_W$, then $\Sh_{L_V\rti K}(V\rti G, \tilde{ \mathfrak{X}}_V)\to \Sh_{L_W\rti K}(W\rti G,\tilde{ \mathfrak{X}}_W)$ is non-zero as a morphism of abelian varieties (for example, using the explicit description of the points over $\co$). Together with Theorem \ref{kings} this demonstrates that $\mu_G^\mot$ is faithful.\todo{}
\todo{I feel like I should make some comment on the fullness of $\mu_G^\mot$. I think it could be in quite a lot of cases but probably isn't in general. I think that it shouldn't be full whenever the generic Mumford-Tate group of $\mu_G^\mot(V)$ is not $G$. (but maybe can't make problems occur in our restricted weights). This is because the motivic decomposition should match that of the VHS, but Mumford showed for $X,Y \in \la V\ra^\ot$:
	\[  \Hom_{\VHS}(X,Y)= \Hom_{\GMT(V)}(X_s,Y_s) \]
	(how is this relevant?) But our $G$ contains $\GMT(V)$ so when this is proper we may get addition morphisms... (still the question of whether this is proper in the case of $\GMT(V)=G$) See Ancona's thesis Sec 9 and Moonen's paper. }

\todo{Not enough time to think about now, but can we also lift tori using mixed Shimura varieties?
	Maybe use something like $(V\rti G)\ti (V\rti G)\to k \rti G$ (but then need $h^2$?).
	In particular, can we deal with $P$ as Ancona does... Obviously this would do something a bit weird as we would want to take $h^1(-)$ of a tori rather than something pure (morally should then have weight $\pm 2$). So we'd need an argument to actually make sense of $h^1(-)$ of a torus but as a Chow motive. Maybe if we could somehow spot within the category of 1-motives the pure objects...
	Another related question: Pink places extra conditions on his mixed Shimura varieties to ensure that the canonical construction provides variations of Hodge structure which are polarisable. Can this be somehow done at the motivic level to get everything Ancona lifts in the PEL case everywhere?}
\end{remark}
\begin{notat}
	 Given $V\in \Rep_F(G)^\AV$, we shall denote the mixed Shimura variety $\Sh_{L\rti K}(V\rti G,\tilde{ \mathfrak{X}})$ simply by $S_{K,V}$. We use $p\colon S_{K,V}\to S$ and $\iota\colon S\to S_{K,V}$ to denote the maps induced by the projection and Levi section as well as the induced maps on their analytifications. We continue accordingly for $(G',h')$.
\end{notat}
\begin{lemma} \label{bccompateasy}
	Given a morphism of Shimura data $f \colon (G',\mathfrak{X}')\to (G,\mathfrak{X})$, a neat open compact $K\le G(\A_f)$, $K'\le G'(\A_f)$ with $f(K')\le K$, and a choice of stable $\hat \z$-lattices for all elements of $\Rep_F(G)$, $\Rep_F(G')$, then the following diagram commutes:
\[
		\begin{tikzcd}
			\Rep_F(G)^\AV \ar{d}[swap]{f^*} \ar{r}{\mu_G^\mot}\arrow[dr,phantom,  "\implies" rotate=-145]  & \CHM_F/S \ar{d}{f^*} \\
			\Rep_F(G')^\AV \ar{r}[swap]{\mu_{G'}^\mot} & \CHM_F/S'
		\end{tikzcd}
\]
	up to a natural isomorphism $\psi \colon f^* \circ  \mu^\mot_{G} \implies \mu^\mot_{G'}  \circ f^*$.
\end{lemma}
\begin{proof}
	From Lemma \ref{fibreprodprop} {\it ii)} and that the canonical projectors defining $h^i$ commute with pullback, we obtain isomorphisms
	\[ (h^1(S_{K,V})^\vee)_{S'}\cong h^1(S'_{K',f^*V})^\vee  .  \]
	The natural isomorphism is then given by taking these maps and possibly composing the maps defined in the proof of Proposition \ref{pain!} if the lattice chosen for $f^*V$ is not $f^*L$.
\end{proof}
\section{Direct images for mixed Shimura varieties}\label{provar}
In this section, we check that $\mu_G^\mot$ lifts the canonical construction and is compatible with base change.\todo{}
\begin{lemma}\label{keylemma}
	Given a Shimura datum $(G,\mathfrak{X})$ and $V\in \Rep_F(G)^\AV$, then there is a canonical identification of $\mu_{G}^\H(V)$ and the dual of $H^1_{B}(S_{K,V}(\co))= R^1p_*F_{S_{K,V}(\co)}$, where $p\colon S_{K,V}(\co)\to S(\co)$ denotes the usual projection.
\end{lemma}
\begin{proof}The canonical construction can be extended to mixed Shimura varieties as we now recall. Let $(P,\tilde{\mathfrak{X}})$ be a mixed Shimura datum and $Q\le P(\A_f)$ a neat open compact subgroup. A representation $W\in \Rep_F(P)$, which we consider as a $\q$-representation $\rho \colon P\to \GL(W)$ together with an $F$-structure, defines a local system
	\[ \mu_P^\H(W):=P(\q)\ba (\tilde{\mathfrak{X}}\ti (P(\A_f)/Q)\ti W)  \]
	on $\Sh_Q(P,\tilde{ \mathfrak{X}})(\co)=P(\q)\ba (\tilde{\mathfrak{X}}\ti (P(\A_f)/Q)) $. Similarly to Construction \ref{cancon}, each fibre $\{ (x,k,v) \mid  v \in W \}\cong W$ has a well-defined mixed Hodge structure given by $\rho \circ h_x$ and $\mu_P^\H(W)$ has the structure of a graded-polarisable variation of Hodge structure. This extends to a $\ot$-functor
	\[ \mu_P^\H \colon \Rep_F(P)\to \VHS_F/\Sh_Q(P,\tilde{ \mathfrak{X}})(\co).  \]\label{mixedcanon}
	This is functorial in the sense that, given $f\colon (P',\mathfrak{X}')\to (P,\mathfrak{X})$ and $Q\le P(\A_f),Q'\le P'(\A_f)$ with $f(Q')\le Q$, then there is a canonical isomorphism
	\[   f^*\mu_{P}^\H(W) =\mu_{P'}^\H(f^*W).\]
	For the purposes of the lemma, the key fact is that pushforwards of sheaves arising via the canonical construction correspond to group cohomology. More specifically, in the notation of the lemma, the following diagram commutes:
\[
		\begin{tikzcd}
			\Rep(V\rti G) \ar{d}[swap]{H^i(V,-)} \ar{r}{\mu_{V\rti G}^\H} & \VHS/S_{K,V}(\co) \ar{d}{R^ip_*}\\
			\Rep(G) \ar{r}[swap]{\mu_G^{\mathcal{H}}} & \VHS/S(\co)
		\end{tikzcd}
\]
	where the left vertical map is group cohomology (see \cite[Thm.\ II.2.3, Prop.\ I.1.6c)]{Wild}). In the case of the one dimensional trivial representation, this yields identifications
	\[   \mu_G^{\mathcal{H}}(H^1(V,F)) = R^1p_*F_{S_{K,V}(\co)}. \]
	But, $H^1(V,F)= V^\vee$ as desired.
\end{proof}
\begin{notat}\label{phidef}
	Write $\phi_V$ for the isomorphism $ H^1_B((S_{{K,V}})(\co))^\vee\overset{\sim}{\to} \mu_G^{\mathcal{H}}(V)$ and $\phi=(\phi_V)_V$ for the collection as $V$ ranges over $V\in \Rep_F(G)^\AV$.
\end{notat}
\begin{lemma}\label{commutativity}
	\begin{enumerate}[i)]
		\item Let $(G,h)$ be a Shimura datum and $\alpha \colon V_1 \to V_2$ a morphism in $\Rep_F(G)^\AV$. Fix a neat open compact subgroup $K \le G(\A_f)$ and let $\alpha$ also denote the map $(S_{K,V_1})(\co)\to (S_{K,V_2})(\co)$. Then the following diagram commutes:
\[
			\begin{tikzcd}
				H^1_B((S_{K,V_1})(\co))^\vee \ar{r}{\phi_{V_1} }\ar{d}[swap]{(\alpha^*)^\vee} & \mu_G^{\mathcal{H}}(V_1) \ar{d}{\mu_G^{\mathcal{H}}(\alpha)}  \\
				H^1_B((S_{K,V_2})(\co))^\vee\ar{r}[swap]{ \phi_{V_2}}&\mu_{G}^\H( V_2)
			\end{tikzcd}
\]
		\item 	Let $f\colon (G',h')\to (G,h)$ be a morphism of Shimura data and $K\le G(\A_f),K'\le G'(\A_f)$ neat open compact subgroups with $f(K')\le K$. For any $V\in \Rep_F(G)^\AV$, the following diagram commutes:
\[			\begin{tikzcd}
				f^*H^1_B(S_{K,V}(\co))^\vee \ar{r}{f^*(\phi_V) }\ar{d}[swap]{H^1_B(\psi_V)} & f^*\mu_G^{\mathcal{H}}(V) \ar{d}{\kappa_V}  \\
				H^1_B((S'_{K',f^*V})(\co))^\vee\ar{r}[swap]{ \phi_{f^*V}}&\mu_{G'}^{\mathcal{H}}( f^*V)
			\end{tikzcd}\]
	\end{enumerate}
\end{lemma}
\begin{proof}
	We prove the first case, the other is similar. The strategy is to reduce to a group theoretic context via a Tannakian argument using work of Wildeshaus. Fix a connected component $S^0$ of $S(\co)$ and let $S_{K,V_i}^0$ denote the connected component $p_i^{-1}(S^0)$. In \cite[Thm.\ II.2.2]{Wild} it is checked that the canonical construction produces variations of Hodge structure which are admissible in the sense of \cite{Kashiwara}. Since the $V_i$ are unipotent, objects in the image of $\mu^\H_{V_i\rti G}$ (in the notation used in the proof of Lemma \ref{keylemma}) admit a filtration by objects pulled back from $S^0$. Let $\VHS'/S^0$ denote the category of admissible variations of Hodge structure on $S^0$ and $p_i\UVar/S_{K,V_i}^0$ denote the full subcategory of $\VHS'/S_{K,V_i}^0$ whose objects admit a filtration for which the graded objects are pulled back from elements of $\VHS'/S^0$. The functors $\mu_G^{\mathcal{H}},\mu^\H_{V_i\rti G}$ take values in these categories.
	
	Fix $y\in S^0$ and for $i=1,2$ set $x_i=\iota_i(y)$, where $\iota_i$ denotes the canonical Levi section. For $i=1,2$, let $P_{i,x_i}$ denote the Tannaka dual of $p_i\UVar/S_{K,V_i}^0$ and $G_y$ the Tannaka dual of $\VHS'/S^0$ all with the obvious fibre functors. The map $P_{i,x_i}\to G_y$ induced by $p_i^*$ is surjective (e.g.\ \cite[Prop.\ 2.21a)]{DeligneMilne}). Lastly, set $V_{i,x_i}=\ker (P_{i,x_i}\to G_{y} )$.
	
	Consider the diagram:
\[
		\begin{tikzcd}[column sep=tiny]
			p_1\UVar/S_{K,V_1}^0\ar{dr}[swap]{R^jp_{1,*}}  &  & \ar{ll}[swap]{\alpha^*}p_2\UVar/S_{K,V_2}^0 \ar{dl}{R^jp_{2,*}} \\
			& \VHS'/S^0&
		\end{tikzcd}
\]
	This does not commute, but there is an obvious natural transformation $R^jp_{2,*}\implies R^jp_{1,*}\alpha^*$. The calculation of higher direct images in $p_i\UVar/S_{K,V_i}$ coincides with the usual higher direct image as elements of $\VHS_F/S_{K,V_i}^0$ (cf.\ \cite[Sec.\ I.4]{Wild}). The maps $R^jp_{i,*}$ are not $\ot$-functors, but we claim that when viewed in the Tannakian setting, the above triangle becomes:
\[		\begin{tikzcd}[column sep=tiny]
			\Rep(P_{1,x_1}) \ar{dr}[swap]{H^j(V_{1,x_1},-)}  &  & \ar{ll}[swap]{\alpha^*}\Rep(P_{2,x_2}) \ar{dl}{H^j(V_{2,x_2},-)} \\
			& \Rep(G_y)&
		\end{tikzcd}\]
	and the natural transformation becomes the usual map \[H^j(V_{2,x_2},-)\implies H^j(V_{1,x_1}, \alpha^*(-)).\] To see this, note that $p_{i}^*$ corresponds to inflation from $G_y$ and has right adjoint $p_{i,*}$, whilst $(-)^{V_{i,x_i}}$ is right adjoint to inflation. 
	
	Since the canonical construction is a $\ot$-functor, after taking duals we obtain a diagram of short exact sequences:
\begin{equation}
		\begin{tikzcd}
			0 \ar{r}& V_{i,x_i}\ar{r}  \ar[dashed]{d}&P_{i,x_i}\ar{r} \ar{d}{t_i}& G_y\ar{r} \ar{d}{r}& 1\\
			0 \ar{r} & V_{i} \ar{r} & P_i \ar{r} & G \ar{r}& 1 
		\end{tikzcd} \label{fundamdiag}
\end{equation}
	where $t_i$ is the dual of $\mu_{V_i\rti G}^\H$ and $r$ the dual of $\mu_G^{\mathcal{H}}$. Moreover, the left vertical map $V_{i,x_i}\to V_i$ is an isomorphism \cite[p.\ 96]{Wild} (this would not be true without restricting to admissible variations of Hodge structure). This shows that the following square commutes:
\begin{equation}
		\begin{tikzcd}
			\Rep(P_i) \ar{r}{t_i^*} \ar{d}[swap]{H^1(V_i,-)} & \Rep(P_{i,x}) \ar{d}{H^1(V_{i,x_i},-)} \\
			\Rep(G) \ar{r}[swap]{r^*}& \Rep(G_y)
		\end{tikzcd}\label{squares}
\end{equation}
	as in the proof of Lemma \ref{keylemma}. In the case of the trivial representation $\q$, this yields maps $r^*H^1(V_i,\q)\to H^1(V_{i,x_i},\q)$ which are dual to $\phi_{V_i}$. Since the diagrams of \eqref{fundamdiag} are compatible with $\alpha^*$, the squares of \eqref{squares} form a prism:
	\[  \begin{tikzcd}[column sep = tiny, row sep=small]
	& \Rep(P_2) \ar[start anchor = {[xshift=-0.5ex,yshift=0.5ex]} ]{dl}[swap]{\alpha^*}  \ar{dddl} \ar{rr}& & \Rep(P_{2,x_2}) \ar[start anchor = {[xshift=-0.5ex,yshift=0.5ex]} ]{dl}[swap]{\alpha^*} \ar{dddl}\\
			\Rep(P_1) \ar[rr,crossing over] \ar{dd} & &  \Rep(P_{1,x}) \ar{dd} &  \\
			 & \\
\Rep(G) \ar{rr}& &\Rep(G_y) &
	\end{tikzcd}   \]
	A purely group theoretic argument now checks that, consequently, there is a commutative square:
\[		\begin{tikzcd}
			H^1(V_{1,x_1},\q)    &\ar{l} r^*H^1(V_1,\q) \\
			H^1(V_{2,x_2},\q) \ar{u}{\alpha^*} &\ar{l} r^*H^1(V_2,\q) \ar{u}[swap]{r^*\alpha^*}
		\end{tikzcd}\]
	Taking Tannaka and linear duals we now obtain the square in {\it i)}.
\end{proof}
\label{compatlift}
We are now able to prove Theorem \ref{intromain} of the introduction.
\begin{theorem}\label{tildenat}
	Let $(G,h)$ be an arbitrary Shimura datum and $K\le G(\A_f)$ neat open compact. Denote by $S$ the Shimura variety $\Sh_K(G,h)$. Then the following diagram commutes,
\[		\begin{tikzcd}[column sep=tiny]
			\Rep_F(G)^\AV \ar{rr}{\mu_G^\mot} \ar{dr}[swap]{\mu_G^\H} && \CHM_F/S \ar{dl}{H^\bullet_B} \ar[dll, phantom, "\implies" rotate=-155, near start, start anchor={[xshift=-2ex]}, end anchor= north ] \\
			{}& \VHS_F/S(\co)
		\end{tikzcd}\]
	up to natural isomorphism given by $ \phi \colon H^\bullet_B\circ \mu_G^{\mot} \implies \mu_G^{\mathcal{H}}$ (where $\phi$ is as in Notation \ref{phidef}). Moreover, under pullback by $f\colon (G',\mathfrak{X}')\to (G,\mathfrak{X})$, the triangles for $(G,\mathfrak{X}),(G',\mathfrak{X}')$ form a commutative prism:
\[		\begin{tikzcd}[column sep=tiny]
			\Rep_F(G)^\AV \ar{dd}[swap]{f^*} \ar{dr}[swap]{\mu_G^{\mathcal{H}}} \ar{rr}{\mu_G^\mot} && \CHM_F/S\ar{dd}{f^*} \ar{dl}{H^\bullet_B}\\
			& \VHS_F/S(\co)  & \\
			\Rep_F(G')^\AV \ar{dr}[swap]{\mu_{G'}^{\mathcal{H}}}\ar[rr,"\mu_{G'}^\mot" near start] && \CHM_F/S' \ar{dl}{H_B^\bullet}\\
			& \VHS_F/S'(\co) \arrow[ uu,crossing over, swap, "f^*" near end, leftarrow] & 
		\end{tikzcd}\]
	for which each face has a given natural transformation, all of which are compatible.
\end{theorem}
\begin{proof}  That $ \phi_V$ defines a natural isomorphism for the first triangle is Lemma \ref{commutativity} {\it i)}. The commutativity of the other individual faces in the prism is given by the natural isomorphisms: $\psi$ of Lemma \ref{bccompateasy} for the rear face, $\kappa$ of Construction \ref{canconbc} for the front left face, and $\xi$ of Remark \ref{xi} for the front right.
	
	Due to O'Sullivan's Theorem \ref{O'Sullivan} (cf.\ Remark \ref{O'Sullivanbc}), we need only prove the compatibility statement for homological motives. 
	As a result, we reduce to showing that the two natural isomorphisms $H^\bullet_B\circ f^* \circ \mu_G^\mot\implies H_B^\bullet \circ\mu_{G'}^\mot\circ f^*$ (which are functors from $\Rep_F(G)\to \VHS_F/S'(\co)$) defined by
	\begin{align*}
	& f^*H^1_B(S_{K,V}(\co))^\vee \overset{H^1_B(\psi_V)}{\to} H^1_B((S'_{K',f^*V})(\co))^\vee, \\
	&f^*H^1_B(S_{K,V}(\co))^\vee \overset{f^*{\phi}_V}{\to} \mu_G^{\mathcal{H}}(V)_{S'(\co)} \overset{\kappa_V^{-1}}{\to} \mu_{G'}^{\mathcal{H}}(f^*V)\overset{\phi_{f^*V}^{-1}}{\to }H^1_B((S'_{K',f^*V})(\co))^\vee ,
	\end{align*}
	coincide, here $\kappa$ is as defined in Construction \ref{canconbc} and $\psi$ is as defined in Lemma \ref{bccompateasy}. This follows from Lemma \ref{commutativity} {\it ii)}.
\end{proof}

\section{Classification of PEL-data} \label{sec:class}
In the case of PEL-type Shimura data, significantly stronger results than Theorem \ref{tildenat} are possible. In this section, we provide a classification of PEL-type Shimura data after base change to $\re$.
\begin{notat} Given an algebra $B/\q$, we write $B_F$ for $B\ot_\q F$. Similarly if $W$ is a $B$-module, then $W_F$ denotes $W\ot_\q F$.
\end{notat}
\begin{definition}	\label{PEL} A \emph{PEL-datum} is a tuple $(B,*,V,\la \ , \ \ra,h)$ consisting of: a semi-simple $\q$-algebra $B$ with a positive (anti-)involution $*$ on $B$, that is an anti-commutative involution such that $\tr_{B_\re/\re}(bb^*)>0$ for all $0\ne b \in B_\re$, together with a finite dimensional $B$-module $V$ equipped with an alternating non-degenerate $\q$-valued pairing $\la \ , \ \ra$ on $V$ such that, for $b \in B, u,v\in V$
	\[  \la bu,v\ra =\la u ,b^*v\ra, \]
	and finally a choice of $\re$-algebra homomorphism $h \colon \co \to \End_{B_\re}(V_\re)$ such that
	\[\begin{cases}\la h(z)u, v \ra =\la u , h(\bar z) v\ra   \quad \forall z\in \co,\  u,v \in V  \\
	\la u,h(i)u\ra \textrm{ is positive definite} \end{cases},  \]
	(the first condition ensures that $\la u ,h(i)v\ra$ is symmetric).
	
	Let $G$ be the algebraic group whose $R$-points, for any $\q$-algebra $R$, are defined by
\[  {G}(R)= \left\{ g \in \Aut_{B_R}(V_R) \ \middle| \ \parbox{7.5cm}{$\exists \, \mu(g) \in R\tii$ such that  $\la gu ,gv \ra = \mu(g) \la u,v\ra$ for all $u,v\in V\ot R$} \right\} .  \]
Note $G$ is connected if and only if $G$ has no factors of ``orthogonal type'' (see Lemma \ref{PELsatisfySV5}). For $z \in \co\tii$, we automatically have that $h(z) \in G(\re)$. We also denote by $h$ the induced map $\S\to G_\re$.
\end{definition}
\begin{notat}
	Any semisimple $\re$-algebra with positive involution splits as a product of simple factors each of which is of one of the following types (see for example \cite[p.\ 386]{Kottwitz}):
	\begin{itemize}
		\item \emph{symplectic}: $(\M_{n}(\re), A\mapsto A^t)$
		\item \emph{linear}: $(\M_{n}(\co),A\mapsto \bar A^t)$, where $(\bar-)$ denotes coefficientwise complex conjugation. 
		\item \emph{orthogonal}: $(\M_{n}(\mathbb{H}), A\mapsto \bar A^t )$, where $(\bar-)$ denotes the (anti-)involution $a+bi+cj+dij\mapsto a-bi-cj-dij$ coefficientwise.
	\end{itemize}\label{Bdecomp}
In particular, all symplectic $B_\re$-modules split as an orthogonal direct sum of submodules only acted on non-trivially by a single simple factor of one of the above types, and $G_{1,\re}$ splits accordingly. 
\end{notat}
\begin{notat}
	Given an algebraic group $G$, we denote by $G^\circ$ the connected component of the identity. We define the following algebraic groups over $\re$:
\begin{itemize}
	\item Let $\U_{a,b}$ be the indefinite unitary group whose $\re$-points consist of elements of $M_{a+b}(\co)$ which preserve a Hermitian form of signature $(a,b)$. There is an obvious isomorphism $\U_{a,b}\cong \U_{b,a}$ and $(\U_{a,b})_\co\cong \GL_{a+b,\co}$.
	\item Set $J=\left(\begin{array}{c c}0&-1\\1 &0 \end{array} \right)$ and let $\O_{2n}^*$ be the algebraic group defined by
	\[ \O_{2n}^*(A)=\left\{g \in \O_{2n}(A\ot \co) \ \ \middle | \ \ \bar g^t\left(\begin{array}{c c c}J & & \\ & \ddots &\\
	& &   J \end{array} \right) g=\left(\begin{array}{c c c}J & & \\ & \ddots &\\
	& & J \end{array} \right)   \right\}, \]
	for an $\re$-algebra $A$. Note that $(\O_{2n}^*)_\co \cong \O_{2n}$.
\end{itemize}
\end{notat}
The following is well-known, but we have been unable to reference explicitly in the literature.
\begin{lemma}\label{PELclass}Let $(B,*,V,\la \ , \ \ra , h)$ be a PEL-datum.
	\begin{enumerate}[i)]
		\item Fix a	decomposition of $(B_\re,*)$ into factors of symplectic, linear and orthogonal types respectively, as in Definition \ref{Bdecomp}, then
		 \[  G_{1,\re} \cong \prod_i \Sp_{2g_i} \ti \prod_j \U_{a_j,b_j} \ti \prod_{k} \O^*_{2r_k},\]
		 with each factor acting on the factor of $V_\re$ for which the action of $B_\re$ factors through the corresponding $\M_n(\re)$, $\M_n(\co)$ or $\M_n(\HH)$. 
		\item If $G_{1,\re}$ has no factors isomorphic to $\U_{n,0}$ for $n\ge 2$, then $(G^\circ,h)$ defines a Shimura datum. In particular, if $G_{1,\re}$ additionally has no factors of orthogonal type, then $(G,h)$ is a Shimura datum.
		\item In any case, the identity connected component of the centre of $G^\circ$ is an almost-direct product of a $\q$-split torus and an $\re$-anisotropic torus.
	\end{enumerate}\label{PELsatisfySV5}\todo{}
\end{lemma}
\begin{proof}
	These properties are well-known, but we provide proofs of the statements we have been unable to find references for. In \cite[Lem.\ 4.1]{Kottwitz} it is shown that $(G,h)$ satifies (1.5.1), (1.5.2) and (1.5.3) of \cite{travauxdeShimura}, even without the assumption of {\it ii)}. To show {\it ii)} it remains to show that $G^\ad$ has no factors of compact type under the above assumption. This and {\it iii)} will be easy to deduce from {\it i)}.

	In order to classify the factors of $G_{1,\re}$ which may arise it suffices to assume that $B_\re$ is simple of each type appearing in Definition \ref{Bdecomp}. Moreover, we are able to reduce to the case of $B_\re$ isomorphic to $\re$, $\co$ or $\mathbb{H}$ by an easy Morita equivalence argument.
	
	We shall make repeated use of the following result of Kottwitz: Let $(C,*)$ be an $\re$-algebra with positive involution and $(W,\la \ , \ \ra , h)$, $(W',\la \ , \ \ra' , h')$ be two triples that together with $(C,*)$ satisfy the conditions of Definition \ref{PEL} with $\re$ in place of $\q$. Then if $W$ and $W'$ are isomorphic as $C \ot_\re \co$-modules, with $\co$ acting via $h$ and $h'$ respectively, then $(W,\la  \ , \ \ra )$ and $(W' \la \ ,\ \ra')$ are isomorphic as symplectic $(C,*)$-modules \cite[Lemma 4.2]{Kottwitz}.
	
	First assume that $(B_\re,*)=(\re,*=\id)$. Then 
	\[\left(W=\re^{\op 2}, \la \ , \ \ra =\left(\begin{array}{rr}
	0 & 1 \\
	-1 & 0
	\end{array}\right), h(i)= \left(\begin{array}{rr}
	0 & -1 \\
	1 & 0
	\end{array}\right) \right)\] is a triple as above with corresponding $B_\re \ot_\re  \co$-module $\co$. As a result, any symplectic $(B_\re,*)$-module $V_\re$ must split as an orthogonal direct sum of terms isomorphic to $W$. By definition, $G_1(\re)$ for $W^{\op n}$ is $\Sp_{2n}$.
	
	Now assume that $(B_\re,*)=(\co,*=z \mapsto \bar z)$. In this case, $B_\re\ot_\re \co\cong \co\ti \co$ has two irreducible modules. The triple given by $(\co,  \tr_{\co/\re}(xi\bar{y}), h(i)= i)$ (resp.\ $(\co,  -\tr_{\co/\re}(xi\bar{y}), h(i)=-i)$) corresponds to the $\co\ot_\re \co$-summand on which the $\co$ actions agree (resp.\ disagree). So if we denote these modules by $A$ and $B$ respectively, then any $(B_\re,*)$-module is isomorphic to $A^{\op a}\op B^{\op b}$. For such a module, $G_1(\re)$ consists of elements of $\GL_n(\co)$ which also preserve a pairing of signature $(b,a)$. In other words, $G_{1,\re}$ is the indefinite unitary group $U_{b,a}$.
	
	Finally, in the quaternion case we shall assume that $(B_\re,*)=(\mathbb{H}^\opp,*)$ (with $\mathbb{H}^\opp$ an expositional choice). Then $B_\re\ot_\re \co \cong \MM_2(\co)$ has a unique non-trivial irreducible module, which is of $\re$-dimension 4. This is realised by the triple $(\mathbb{H},\tr_{\mathbb{H}/\re}(xj\tilde{y}),h(i)=j )$ where $\mathbb{H}^\opp$ acts by right multiplication and $y\mapsto \tilde{y}$ is the (anti-)involution given by $y=a+bi+cj+dij\mapsto a+bi-cj+dij$. As such, all symplectic $(\HH^\opp, *)$-modules are isomorphic to $\HH^{\op n}$ for some $n$. In this case, $\End_{\HH^\opp}(\HH^{\op n})\cong \M_n(\HH)$, with $\HH$ acting by left multiplication and taking adjoints with respect to the pairing coincides with $A\mapsto \tilde{A}^t$. If we embed $\HH\hookrightarrow M_2(\co)$ by $i\mapsto \left(\begin{array}{rr}
	i & 0 \\
	0 & -i
	\end{array}\right)$ and $j\mapsto \left(\begin{array}{rr}
	0 & 1 \\
	-1 & 0
	\end{array}\right)$ and extend this to an embedding $\M_n(\HH) \hookrightarrow \M_{2n}(\co)$, then matrix transposition restricts, on the image of $\M_n(\HH)$, to taking adjoints. As a result, we may view $G_1$ as the algebraic group whose $\re$-points consists of elements of $\O_{2n}(\co)$ which lie in the image of $\M_n(\HH)$. Since the image of $\M_n(\HH)$ consists of matrices $g$ for which $\bar g \diag(J,...,J)=\diag(J,...,J) g$, these are precisely elements of $\O^*_{2n}(\re)$.
	
	To deduce {\it ii)}, note that $G_1^\ad \cong G^\ad$ (indeed, the cokernel of $G_1^\ad\hookrightarrow G^\ad$ is a proper quotient of $\G_m$). From the above calculations we find that the only possible factors of $G_1^\ad$ of compact type are $U_{n,0}\cong U_{0,n}$ for $n\ge 2$. For {\it iii)}, first note that the largest anisotropic subtorus of $Z(G^\circ)$ must be contained in $Z(G_1^\circ)$. But from the above calculation $(Z(G_1^\circ)^\circ)_\re$ is always anisotropic. Indeed, $Z(\Sp_{2g})$ is finite whilst $Z(\U_{a,b}), Z(\SO_{2n}^*)\cong Z(\SO_2^*)$ are both isomorphic to $U_1$.
\end{proof}
	The factorisation of Lemma \ref{PELsatisfySV5} {\it i)} justifies the naming convention of Definition \ref{Bdecomp}.
	\begin{remark}
		In \cite{Kottwitz} Kottwitz, allows Shimura data to have (not necessarily connected) reductive groups $G$ of the form considered in Lemma \ref{PELsatisfySV5} when  $G_{1,\re}$ has no factors isomorphic to $U_{n,0}$ for $n\ge 2$. Ancona's results also hold in this generality and so ours will as well.
	\end{remark}
 \begin{definition}A Shimura datum $(G,h)$ which arises as in Lemma \ref{PELsatisfySV5} is said to be of \emph{PEL-type} and the corresponding $(B,*,V,\la \ , \ \ra,h)$ is said to be a \emph{PEL-datum} for $(G,h)$.
 
 If we fix such a PEL-datum for $(G,h)$, then we say that $V\in \Rep(G)$ is the \emph{standard representation} of $G$. Shimura data with a fixed choice of PEL-datum have an explicit moduli interpretation (see for example \cite[Sec.\ 8]{MilneSV}).\end{definition}
\begin{example}\label{twoPELdata}From the proof of Lemma \ref{PELsatisfySV5}, it is easy to see that a Shimura datum may admit multiple distinct PEL-data due to Morita equivalence. As an explicit example, consider the PEL-datum $\left(\q,*,\q^{\op 2},\left(\begin{array}{rr}
	0 & 1 \\
	-1 & 0
	\end{array}\right),h(a+bi)=\left(\begin{array}{rr}
	a & -b \\
	b & a
	\end{array}\right)\right)$, which corresponds to the usual modular curves Shimura datum $(\GL_2,\mathcal{H})$.
	
	There is also a PEL-datum $\left(\MM_2(\q),*=(-)^t, \q^{\op 4},\left(\begin{array}{r|r}
	0 & I_2 \\
	\hline
	-I_2 & 0
	\end{array}\right),h(a+bi) = \left(\begin{array}{r|r}
	aI_2 & -bI_2\\
	\hline
	bI_2 & aI_2
	\end{array}\right) \right)$, where $M_2(\q)$ acts diagonally on $\q^{\op 4} = \q^{\op 2}\op \q^{\op2}$. Then $G$ is isomorphic to $\GL_2$ embedded within $\GL_4(\q) $ via $ \left(\begin{array}{rr}
	a & b \\
	c & d
	\end{array}\right)\mapsto \left(\begin{array}{r|r}
	aI_2 & bI_2 \\
	\hline
	cI_2 & dI_2
	\end{array}\right)$, so that the associated Shimura datum is again $(\GL_2,\mathcal{H})$.
\end{example}

\section{Ancona's construction}\label{Ancconst}
In the case of PEL-type Shimura data, Ancona has described a lift of $\mu_G^{\mathcal{H}}$ defined on all of $\Rep_F(G)$ \cite{Anconapaper}. But, as defined, it is not immediately clear that it is well behaved with respect to pullbacks or is even independent of the choice of PEL-datum (cf.\ Example \ref{twoPELdata}). In this section we briefly recall Ancona's construction, but in the language of mixed Shimura varieties.

\begin{lemma}\label{DM}
	Given a Shimura datum $(G,h)$ with a choice of PEL-datum $(B,*,V,\la \ , \ \ra,h)$, then for all fields $F/\q$, all objects of $\Rep_F(G)$ are, up to isomorphism, direct summands of some space of the form $\bigoplus_{i=1}^k V_F^{\ot a_k}\ot V_F^{\ot b_k}$ (with $V_F= V\ot_\q F$).
\end{lemma}
\begin{proof}
	As $V$ is a faithful $G$-representation, this follows from the proof of \cite[Prop.\ 2.20]{DeligneMilne}.
\end{proof}
\begin{theorem}[{\textrm{\cite[Prop.\ 8.5]{Anconapaper}}}] \label{AnconaEnd}
	Given a Shimura datum $(G,h)$ with a PEL-datum $(B,*,V,\la \ , \ \ra,h)$, let $K$ be a neat open compact subgroup of $G(\A_f)$ and $L$ a $\hat \z$-lattice of $V_F$ (considered as a representation over $\q$). Then for any $n \in \n$, there is a canonical inclusion of rings $a\colon  \End_{\Rep_F(G)}(V_F^{\ot n})\hookrightarrow \End_{\HomM_F/S}(h^1(S_{V_F,K})^{\vee \ot n})$ such that the diagram
\[		\begin{tikzcd}[column sep=tiny]
			\End_{\Rep_F(G)}(V_F^{\ot n})\ar{dr}[swap]{\mu_G^\H} \ar{rr}{a} && 	\End_{\HomM/S}(h^1(S_{V_F,K})^{\vee \ot n}) \ar{dl}{H^\bullet_B}  \\
			& \End_{\VHS/S(\co)}(\mu_G^{\mathcal{H}}(V_F)^{\ot n})
		\end{tikzcd}\]
	commutes. Here, we have used the isomorphism $ \phi_{V_F} \colon H^1_B((S_{V_F,K})(\co))^\vee\to \mu_G^{\mathcal{H}}(V_F)$ of Lemma \ref{keylemma} to identify $\End(\mu_G^{\mathcal{H}}(V_F)^{\ot n})$ and $\End(H^1((S_{V_F,K})(\co))^{\ot n})$.
\end{theorem}
\begin{remark}
	Ancona's strategy is to lift endomorphisms of $V_F$ itself (in our presentation, this is via functoriality of mixed Shimura varieties) and permutations of $V_F^{\ot n}$ in the obvious way, and then additionally lift cycles arising from the polarisation via Poincar\'e duality and Hard Lefschetz (both of which have a motivic interpretation). Ancona then shows that endomorphisms of the above kinds generate all of $\End_{\Rep_F(G)}(V_F^{\ot n})$ in the case of PEL-type Shimura varieties. Whilst Ancona's result allows for Shimura data corresponding to orthogonal groups, the analogous statement does not hold for special orthogonal groups, which would require lifting additional cycles.
\end{remark}
\begin{constr}[{\cite[pf.\ of Thm.\ 8.6]{Anconapaper}}]\label{Anconaconstr}
	There is a $\ot$-functor $\Anc_G \colon \Rep_F(G)\to \HomM_F/S$ defined as follows: set $\Anc_G (V_F^{\ot n})=h^1(S_{V_F,K})^{\vee \ot n}$ and let $\Anc_G(\alpha)$ for $\alpha \in \End (V_F^{\ot n})$ be defined via the map of Theorem \ref{AnconaEnd}. By Hom-tensor adjunction, Theorem \ref{AnconaEnd} also defines a motivic lift of the map $\one \to V\ot V^\vee$. More generally, to define the image of elements of $\Hom(V_F^{\ot a}\ot V_F^{\vee \ot b},V_F^{\ot c} \ot V_F^{\vee \ot d})$ it suffices to fix the image of $\Hom(V_F^{\ot (a+ d)},V_F^{\ot (b+c) })$, but for weight reasons this is zero unless $a-b=c-d$, in which case it is covered by Theorem \ref{AnconaEnd}. 
	
	This also allows us to define, for any choice of idempotent $e$, the image of a direct summand $e\cdot (\bigoplus V_F^{\ot a_n}\ot V_F^{\vee \ot b_n})$. Since every element of $W\in \Rep_F(G)$ is of this form by Lemma \ref{DM}, if we pick a fixed isomorphism $\theta_W \colon W\overset{\sim}{\to} e_W\cdot (\bigoplus V_F^{\ot a_{W,n}}\ot V_F^{\vee \ot b_{W,n}})$ for each $W$, then we can compatibly extend $\Anc_G$ to all of $\Rep_F(G)$. Finally, by composition with the section of Theorem \ref{O'Sullivan}, we obtain a functor $\Rep_F(G)\to \CHM_F/S$, which we also denote $\Anc_G$.\todo{}
\end{constr}

\begin{lemma}\label{indep}
	The construction of $\Anc_G$ is, up to natural isomorphism, independent of all choices made.
\end{lemma}%
\begin{proof}
	Fix $W \in \Rep_F(G)$ and two summands of tensor spaces, $e \cdot\bigoplus V_F^{a_k}\ot V_F^{\vee \ot b_k}$, $e'\cdot\bigoplus V_F^{\ot a'_k}\ot V_F^{\vee \ot b'_k}$, which are both isomorphic to $W$. We must provide an isomorphism 
	\[e \cdot\bigoplus h^1(S_{V_F,K})^{\vee a_k}\ot h^1(S_{V_F,K})^{\ot b_k} \to e' \cdot\bigoplus h^1(S_{V_F,K})^{\vee a'_k}\ot h^1(S_{V_F,K})^{\ot b'_k}.\]
	 Given the compatibility of the K\"unneth formula with mixed Shimura varieties, we may assume that $W$ is irreducible and there is a corresponding isomorphism $e\cdot (V_F^{\ot a}\ot V_F^{\vee \ot b})\to e\cdot (V_F^{\ot a'} \ot V_{F}^{\vee \ot b'})$. 
	 
	 As before, it suffices to assume that $b=b'=0$. For weight reasons, we must then have that $a=a'$.	Finally, since Lemma \ref{AnconaEnd} lifts all elements of $\End_{\Rep_F(G)}(V_F^{\ot a})$, we obtain a motivic lift of the isomorphism between the two tensor space representatives of $W$. For varying $W$, this yields a natural isomorphism and so the desired independence.
\end{proof}
\begin{remark}\label{Anconaextends}
	Let $(G,\mathfrak{X})$ be a Shimura datum with a chosen PEL-datum for which all objects of $\Rep(G)^\AV$ are direct summands of $V^{\op n}$ for varying $n$. Then the argument given above can be adapted to show that $\Anc_G$ extends $\mu^\mot_G$ up to natural isomorphism. If the PEL-datum only has factors of symplectic type in the sense of Definition \ref{Bdecomp}, then this always holds (see Lemma \ref{AVsummand}). This can also be checked to hold much more generally.
\end{remark} 
\begin{theorem}[\textrm{\rm \cite[Thm.\ 8.6]{Anconapaper}}]\label{ancfunc} Let $(G,h)$ be a Shimura datum of PEL-type with a fixed PEL-datum $(B,*,V,\la \ , \ \ra,h)$. Fix also a choice of neat open compact subgroup $K\le G(\A_f)$ and denote by $S$ the Shimura variety $\Sh_K(G,h)$. Then the following diagram commutes,
\[		\begin{tikzcd}[column sep=tiny]
		\Rep_F(G) \ar{rr}{\Anc_G} \ar{dr}[swap]{\mu_G^\H} && \CHM_F/S \ar{dl}{H^\bullet_B} \ar[dll, phantom, "\implies" rotate=-155, near start, start anchor={[xshift=-2ex]}, end anchor= north ] \\
		{}& \VHS_F/{S(\co)}
		\end{tikzcd}\]
	up to canonical natural isomorphism.
\end{theorem}
\begin{proof}
 The natural isomorphism necessarily depends on the choice of $\Anc_G$. Explicitly, in the notation of Construction \ref{Anconaconstr}, write $\eta_{G,V}$ for
	\[  \mu_G^\H(\theta_W^{-1}) \circ ( e_W\cdot \bigoplus (\phi_{V_F}^{\ot a_{W,n}} \ot \phi_{V_F}^{\vee\ot b_{W,n}})),  \]
	where $\phi_{V_F}$ is as defined in Notation \ref{phidef}. That $\eta_G:=(\eta_{G,V})_V$ defines a natural isomorphism now follows from Lemma \ref{commutativity} {\it i)}.
\end{proof}
\section{Compatibility with base change}
In this section, we give conditions to ensure Ancona's construction and Theorem \ref{ancfunc} are compatible with base change, i.e.\ there is a commutative prism analogous to that of Theorem \ref{tildenat}.

Let $f\colon (G',h')\to (G,h)$ be a morphism of Shimura data each with a chosen PEL-datum. Denote their standard representations by $V',V$ respectively. By Lemma \ref{DM}, $f^*V \cong e\cdot (\bigoplus_n (V^{\ot a_n} \ot V^{\vee \ot b_n}))$. In order to show that $\Anc_{(-)}$ is compatible with $f$, we would need to construct an isomorphism
\[ f^*(h^1(S_{K,V})^\vee)\overset{\sim}{\longrightarrow}  e \cdot \left(\bigoplus h^1(S_{K',V'})^{\vee \ot a_n}\ot h^1(S_{K',V'})^{\ot b_n}\right) . \]
Unfortunately, such a morphism cannot be constructed using just functoriality of mixed Shimura varieties. For this reason we make the following restriction:
\begin{definition}\label{admissibledef}
	Let $f \colon (G',h')\to (G,h)$ be a morphism of PEL-type Shimura data each with a choice of PEL-datum with standard representations $V',V$. If
	\[  (\star) \qquad  f^*V \cong e\cdot V'^{\op n} \text{ for some $n\in \n$ and idempotent }e\in \End_{\Rep(G')}(V'^{\op n}), \label{star}\]
	then we say that $f$ is an \emph{admissible} morphism of Shimura varieties with PEL-data. 
	
	Note that if $f$ is admissible, then $f^*V_F \cong e_F \cdot V_F'^{\op n}$ for any $F$. Admissibility implies that there is exists a map $(S_{K,V})_{S'}\to \prod_{i=1}^n S'_{K',V'}$ as abelian varieties over $S'$.
\end{definition}
\begin{example}Given a PEL-datum $(B,*,V, \la \ , \ \ra , h)$ and $B\subseteq B'$ a $\q$-subalgebra, then $(B',*,V,\la \ , \ \ra, h)$ is also a PEL-datum. If $(G,h), (G',h)$ denote the respective Shimura data, then the induced map $(G',h)\hookrightarrow (G,h)$ with the above choices is an admissible morphism.%
\end{example}
\begin{lemma}\label{identityadmissible}
	The identity map $(G,\mathfrak{X})\to (G,\mathfrak{X})$ is admissible for any choice of PEL-data for the source and target.
\end{lemma}
\begin{proof}
	Let $V',V$ be the standard representations of the source and target respectively and $B',B$ the chosen $\q$-algebras. It suffices to show that $V_\re$ is a summand of some $V_\re'^{\op n}$. It is a consequence of Lemma \ref{PELsatisfySV5} {\it i)} that the pairs $(B_\re,V_\re)$ and $(B'_\re,V'_\re)$ may only differ up to Morita equivalence (given that they both correspond to $G_{1,\re}$). To be more explicit, say $B_\re$ has a factor $M_a(\HH)$ with corresponding factor $(\HH^{\op a})^{\op n}$ of $V_\re$, then $B'_{\re}$ has a factor $M_b(\HH)$ with corresponding factor $(\HH^{\op b})^{\op n}$ of $V'_{\re}$. The corresponding factor of $G_{1,\re}$ is then $\O_{2n}^*$ acting in the obvious way. It is then clear that $V_\re$ is a summand of some number of copies of $V'_{\re}$ as $G_{\re}$-modules.
\end{proof}
\begin{example}%
	In Example \ref{twoPELdata}, we described two PEL-data for $(\GL_2,\mathcal{H})$, one with standard representation $V'=\q^{\op 2}$ and the other with standard representation $V=\q^{\op 2}\op \q^{\op 2}$. The identity map $(\GL_2,\mathcal{H})\to (\GL_2,\mathcal{H})$ is admissible for each of the two ways of assigning each $(\GL_2,\mathcal{H})$ a distinct choice of the two PEL-data. Indeed, $\id^*V'\cong(i_1\circ\pi_1)\cdot V$ and $\id^*V\cong V'^{\op 2}$.
\end{example}
Not all morphisms of Shimura data with chosen PEL-data are admissible (see Example \ref{admissibilitycounterexample}), but in Section \ref{sec:symplectic-case} we show that if the PEL-datum on the source has only factors of symplectic type then it is admissible. In any case, it is easy to check if a given morphism is admissible.

We now assume $f\colon (G',h') \to (G,h)$ is admissible and fix one such isomorphism as in $(\star)$.
\begin{constr}
	We now have canonical isomorphisms:
	\begin{align*}  f^*\Anc_G(V)& = h^1(S_{K,V})^\vee_{S'}\\
	&=h^1((S_{K,V})_{S'})^\vee,\\ 
	\intertext{using that the canonical projectors $h^i$ commute with pullbacks \cite[Thm.\ 3.1]{DeningerMurre},}
	&\hspace{-3.5mm}\overset{\text{Lem.\ \ref{fibreprodprop}}}{=} h^1(S'_{K',f^*V})^\vee\\
	&\hspace{-0.5mm}\overset{(\star)}{=}h^1(S'_{K',e\cdot V'^{\op n}})^\vee\\
	&=e\cdot (h^1(S'_{K',V'})^{\op n})^\vee\\
	&= \Anc_{G'}f^* V.
	\end{align*}	 
	by Lemma \ref{fibreprodprop} {\it i)} and the K\"unneth formula \ref{kunneth}.
	Write $\lambda_{V}$ for this composite. For $V_F$, there is an analogous $\lambda_{V_F}$.
\end{constr}
\begin{notat}
	As functors on $\Rep_F(G)$, we extend this to a putative natural isomorphism $\lambda \colon f^*\circ\Anc_G\implies \Anc_{G'}\circ f^*$ as follows: Let $W \in \Rep_F(G)$. Since the construction of $\Anc_{G'}$ is independent of the choice of the $\theta'_{W'}$ (Lemma \ref{indep}), we are free to assume that, for $W\in \Rep_F(G)$ with $\theta_W \colon W \overset{\sim}{\to} e_W\cdot (\bigoplus V_F^{\ot a_n}\ot V_F^{\vee \ot b_n}) $, then $\theta'_{f^*W}$ is obtained from $f^*\theta_W$ by taking the tensor products and direct sums of (the base change of) the isomorphism of $(\star)$. So we have,
	\begin{align*} f^*\Anc_G(W)&= e_W\cdot(\bigoplus h^1(S_{K,V_F})_{S'}^{ \vee \ot a_k}\ot h^1(S_{K,V_F})_{S'}^{ \ot b_k}),\\ \Anc_{G'}f^*(W) &= e_W\cdot( \bigoplus (e\cdot\bigoplus h^1(S'_{K',V'_F}))^{\vee\ot a_k}\ot (e\cdot\bigoplus h^1(S'_{K',V'_F}))^{ \ot b_k}) ) .\end{align*}
	There is now an obvious choice for $\lambda_W$ given by taking sums and products of $\lambda_{V_F}$ and its dual.
\end{notat}
\begin{theorem}\label{ancpullback}\label{main}Let $f\colon (G',h')\to (G,h)$ be an admissible morphism of PEL-type Shimura data with fixed PEL-data. Then
	\begin{enumerate}[i)]
		\item the following diagram commutes:
		\[	\begin{tikzcd}
		\Rep_F(G) \ar{d}[swap]{f^*} \ar{r}{\Anc_G}\arrow[dr,phantom,  "\implies" rotate=-145]  & \CHM_F/S \ar{d}{f^*} \\
		\Rep_F(G') \ar{r}[swap]{\Anc_{G'}} & \CHM_F/S'
		\end{tikzcd}\]
		up to natural isomorphism given by $\lambda\colon f^*\circ \Anc_{G'} \implies \Anc_{G'}\circ f^*$.
		\item there is a commutative prism
			\[	\begin{tikzcd}[column sep=tiny]
		\Rep_F(G) \ar{dd}[swap]{f^*} \ar{dr}[swap]{\mu_{G}^{\mathcal{H}}} \ar{rr}{\Anc_G} && \CHM_F/S\ar{dd}{f^*} \ar{dl}{H^\bullet_B}\\
		& \VHS_F/S(\co)  & \\
		\Rep_F(G') \ar{dr}[swap]{\mu_{G'}^{\mathcal{H}}}\ar[rr,"\Anc_{G'}" near start] && \CHM_F/S' \ar{dl}{H_B^\bullet}\\
		& \VHS_F/S'(\co) \arrow[ uu,crossing over, swap, "f^*" near end, leftarrow] & 
		\end{tikzcd}\]
		for which the prescribed natural isomorphisms on each face are compatible.
		\end{enumerate}
\end{theorem}
\begin{proof}
	For {\it i)}, we must check that $\lambda$ is a natural isomorphism. In view of Theorem \ref{O'Sullivan} (whose section is used in Construction \ref{Anconaconstr} to define $\Anc_G$), it suffices to check commutativity after projection to homological motives. Moreover, since the functor $H^\bullet_B(-)$ is injective on $\Hom_{\HomM_F/S'}(h^i(A_1),h^i(A_2))$ for $A_1, A_2$ abelian varieties over $S'$ (see Remark \ref{faithful}), it is enough to check that $H^\bullet_B(\lambda) \colon H^\bullet_B\circ f^*\circ\Anc_G \implies H^\bullet_B\circ \Anc_{G'}\circ f^*$ is a natural isomorphism. But we already have a natural isomorphism $ H^\bullet_B \circ f^*\circ\Anc_G \implies H_B^\bullet\circ \Anc_{G'}\circ f^*$, namely by composing the realisations of the natural isomorphisms of the other faces appearing the prism of {\it ii)} (this doesn't use the compatibility statement of {\it ii)}). So it suffices to check that $H^\bullet_B(\lambda)$ coincides with the one already constructed. We need only check this for $V_F$ itself, i.e. that
	\[	H_B^\bullet(\lambda_{V_F})= \eta_{S',f^*V_F}^{-1}\circ\kappa_{V_F}\circ f^*(\eta_{S,V_F})\circ  \xi_{h^1(S_{K,V_F})}^{-1} .\]
	Here $\eta_{S,V_F}$ is as defined in the proof of Theorem \ref{ancfunc}, $\kappa$ is as defined in Construction \ref{canconbc} and $\xi$ is as defined in Remark \ref{xi}.
	
	Applying $\xi_{h^1(S_{K,V_F})}$ to both sides, this means checking the equality of:
	\begin{align*}
	& f^*H^1_B((S_{K,V_F})(\co)) {\to}H^1((S_{K',f^*V_F})(\co) ) \overset{(\star)}{\to} e \cdot  H_B^1((S'_{K',V'_F})(\co)), \\
	&f^*H_B^1((S_{K,V_F})(\co)) \overset{f^*{\phi}_{V_F}}{\to} f^*\mu_G^{\mathcal{H}}(V_F) \overset{\kappa_V}{\to} \mu_{G'}^{\mathcal{H}}(f^*V_F) \overset{ \phi_{f^*V_F}^{-1}}{\to }H^1_B((S_{K,f^V_F})(\co))\\
	&\hspace{9cm} \overset{(\star)}{\to}   e\cdot\bigoplus H^1_B((S'_{K',V'_F})(\co)) ,
	\end{align*}
	where, in the second line the composite of the last two maps is $\eta_{S',f^*V_F}^{-1}$, as defined in Theorem \ref{ancfunc}. The equality now follows from the commutativity of:
\[		\begin{tikzcd}
			f^*H^1_B((S_{K,V_F})(\co)) \ar{r}{f^*(\phi_{V_F}) }\ar{d}[swap]{H^1(\lambda_{V_F})} & f^*\mu_G^{\mathcal{H}}(V_F) \ar{d}{\kappa_{V_F}}  \\
			H^1_B((S'_{K',f^*V_F})(\co))\ar{r}[swap]{ \phi_{f^*V_F}}&\mu_{G'}^{\mathcal{H}}( f^*V_F)
		\end{tikzcd}\]
	as shown in Lemma \ref{commutativity} {\it ii)}.
	
	In proving {\it i)} we, in fact, verified the compatibility statement of {\it ii)}.
	\todo{}
\end{proof}
Note that the statement of Theorem \ref{ancpullback} {\it i)} is independent of the choice of realisation. Since the identity map is always admissible (Lemma \ref{identityadmissible}), this shows that Ancona's construction is independent of the choice of PEL-datum. 
\section{\'Etale canonical construction}\label{etalecase}
Canonical constructions arise more generally than just the Hodge realisation, and both $ \mu_G^\mot$ and Ancona's construction should also be lifts of any such construction. We sketch this for the \'etale realisation following \cite[Sec.\ II.4]{Wild}. We use the notation of the \'etale realisation described in Lemma \ref{etalerealisation} \todo{}.

\begin{notat}
	Let $(G,\mathfrak{X})$ be a Shimura datum and $K$ be a neat open compact subgroup of $G(\A_f)$. We consider $S:=\Sh_K(G,\mathfrak{X})$ to be defined over its reflex field $E/\q$ via canonical models. Let $V\in \Rep_F(G)$ and $L$ be a $K$-stable full rank $\hat \z$-sublattice of $V(\A_f)$. Recall from Section \ref{msv} that there is a mixed Shimura variety $S_{K,V}:=\Sh_{L\rti K}(V\rti G,\tilde{\mathfrak{X}})$ whose reflex field is the same as that of $S$.
	The projection and Levi section then define regular maps $p\colon S_{K,V}\to S, \iota \colon   S\to S_{K,V}$.%
\end{notat}
\begin{constr}
	Let $(G,\mathfrak{X})$ be a Shimura datum and $K\le G(\A_f)$ neat open compact. If $K'\le K$ is an open normal subgroup, then there is a right action of $K/K'$ on $\Sh_{K'}(G,\mathfrak{X})$. Since we are assuming that the centre of $G$ is an almost-direct product of a $\q$-split and $\re$-anisotropic torus, the action of $K/K'$ is free on $\co$-points and
	\[ \Sh_{K'}(G,\mathfrak{X}) \longrightarrow\Sh_{K}(G,\mathfrak{X})  \]
	is an \'etale cover of smooth algebraic varieties with Galois group $K/K' $ (see \cite[Prop.\ 3.3.3. and (3.4.1)]{PinkSheaves}).
	 
 Taking the inverse limit over $K'\le K$ we obtain a pro-Galois covering of $\Sh_K(G,\mathfrak{X})$ with Galois group $K$. Then, in the notation of Section \ref{sec:realisations}, any $F_\lambda$-linear continuous representation of $K$ will define a lisse $F_\lambda$-sheaf on $\Sh_K(G,\mathfrak{X})$.
	
	Given $(G_F\overset{\rho}{\to} \GL(V))\in \Rep_F(G)$, we obtain such a representation via
	\[K\hookrightarrow G(\A_f) \twoheadrightarrow G(\q_\ell) \hookrightarrow G(F_\lambda)=G_F(F_\lambda)\overset{\rho(F_\lambda)}{\to } \GL(V)(F_\lambda) . \]
	This defines a functor
	\[ \mu_{G}^\et \colon \Rep_F(G)\to \Et_{F_\lambda}/ S, \]
	which we refer to as the \emph{\'etale canonical construction}.%
\end{constr}
\begin{lemma}
	Given a Shimura datum $(G,\mathfrak{X})$ and $V\in \Rep_F(G)^\AV$, then there is a canonical identification $\phi_{V,\lambda}\colon H^1_{\lambda}(S_{K,V})^\vee \overset{\sim}{\to}\mu_G^\et(V)$. 
\end{lemma}
\begin{proof}
	The \'etale canonical construction extends  verbatim to mixed Shimura varieties. As in the Hodge case, the diagram 
\[		\begin{tikzcd}
			\Rep_F(V\rti G) \ar{d}[swap]{H^i(V,-)} \ar{r}{\mu_{V\rti G}^\et} & \Et_{F_\lambda}/S_{K,V} \ar{d}{R^ip_*}\\
			\Rep_F(G) \ar{r}[swap]{\mu_{G}^\et} & \Et_{F_\lambda}/S
		\end{tikzcd}\]
	commutes \cite[Thm.\ II.4.7, Thm. I.4.3]{Wild}. The dual of the desired isomorphism is given by commutativity in the case of the trivial representation $F$.
\end{proof}
\begin{lemma}
	\begin{enumerate}[i)]
		\item  Let $(G,\mathfrak{X})$ be a Shimura datum and $\alpha \colon V_1 \to V_2$ a morphism in $\Rep_F(G)^\AV$. Fix a neat open compact subgroup $K \le G(\A_f)$ and let $\alpha$ also denote the map $S_{K,V_1}\to S_{K,V_2}$. Then the following diagram commutes:
	\[		\begin{tikzcd}
				H^1_{\lambda}(S_{K,V_1})^\vee \ar{r}{\phi_{V_1,\lambda} }\ar{d}[swap]{(\alpha^*)^\vee} & \mu_{G}^\et(V_1) \ar{d}{\mu_G^{\mathcal{H}}(\alpha)}  \\
				H^1_{\lambda}(S_{K,V_2})^\vee\ar{r}[swap]{ \phi_{V_2,\lambda}}&\mu_{G}^\et( V_2)
			\end{tikzcd}\]
		\item 	Let $f\colon (G',\mathfrak{X}')\to (G,\mathfrak{X})$ be a morphism of Shimura data and $K\le G(\A_f),K'\le G'(\A_f)$ neat open compact subgroups for which $f(K')\le K$. Write $E'$ for the reflex field of $(G',\mathfrak{X'})$ (so $E' \supseteq E$). For any $V\in \Rep_F(G)^\AV$, the following diagram commutes:
\[			\begin{tikzcd}
				f^*(H^1_\lambda(S_{K,V})^\vee)_{E'} \ar{r}{f^*(\phi_{V,\lambda}) }\ar{d} & f^*\mu_G^\et(V)_{E'} \ar{d}  \\
				H^1_\lambda(S'_{K',f^*V})^\vee\ar{r}[swap]{ \phi_{f^*V,\lambda}}&\mu_{G'}^\et( f^*V)
			\end{tikzcd}\]
		Here, on the top row, $f^*$ denotes pullback via the map $S_{K',f^*V}\to (S_{K,V})_{E'}$ and $(-)_{E'}$ pullback via $(S_{K,V})_{E'}\to S_{K,V}$.
	\end{enumerate}
\end{lemma}
\begin{proof}
	As for Lemma \ref{commutativity}, but using \cite[Cor.\ I.3.2 i)]{Wild}.
\end{proof}

We now obtain results analogous to Theorem \ref{tildenat} and Lemmas \ref{ancfunc}, \ref{ancpullback}, whose proofs are near enough identical.
\begin{lemma}\label{tildenatet}
	Let $(G,\mathfrak{X})$ be an arbitrary Shimura datum and $K\le G(\A_f)$ neat open compact. Denote by $S$ the Shimura variety $\Sh_K(G,\mathfrak{X})$. Then the following diagram commutes,
\[		\begin{tikzcd}[column sep=tiny]
			\Rep_F(G)^\AV \ar{rr}{\mu_G^\mot} \ar{dr}[swap]{\mu_G^\et} && \CHM_F/S \ar{dl}{H^\bullet_\lambda} \ar[dll, phantom, "\implies" rotate=-155, near start, start anchor={[xshift=-2ex]}, end anchor= north ] \\
			{}& \Et_{F_\lambda}/S
		\end{tikzcd}
\]
	up to natural isomorphism given by $ \phi \colon H^\bullet_{\lambda}\circ \mu_G^{\mot} \implies \mu_G^{\mathcal{H}}$. Moreover, under pullback by $f\colon (G',\mathfrak{X}')\to (G,\mathfrak{X})$, the triangles for $(G,\mathfrak{X}),(G',\mathfrak{X}')$ form a commutative prism for which the given natural transformations on each face are compatible.
\end{lemma}
\begin{lemma}
	\begin{enumerate}[i)]
		\item	 Let $(G,h)$ be a Shimura datum of PEL-type with a fixed PEL datum $(B,*,V,\la \ , \ \ra,h)$. Fix also a choice of neat open compact subgroup $K\le G(\A_f)$ and denote by $S$ the Shimura variety $\Sh_K(G,h)$. Then the following diagram commutes,
\[			\begin{tikzcd}[column sep=tiny]
				\Rep_F(G) \ar{rr}{\Anc_G} \ar{dr}[swap]{\mu_G^\et} && \CHM_F/S \ar{dl}{H^\bullet_\lambda} \ar[dll, phantom, "\implies" rotate=-155, near start, start anchor={[xshift=-2ex]}, end anchor= north ] \\
				{}& \Et_{F_\lambda}/S
			\end{tikzcd}\]
		up to canonical natural isomorphism.
		\item Given a morphism of Shimura data $f \colon (G',h')\to (G,h)$, each of PEL-type with a fixed datum, which is admissible in the sense of Definition \ref{admissibledef}, then the triangles for $(G,h)$ and for $(G',h')$ together with base change form a commutative prism as in Theorem \ref{main}. Each face has a prescribed natural isomorphism which altogether are compatible.
	\end{enumerate}
\end{lemma}

\section{Results on admissibility}\label{sec:symplectic-case}
In this section, we give additional results on the admissibility of morphisms of Shimura data with chosen PEL-data. Firstly, not all such morphisms are admissible:
\begin{example}\label{admissibilitycounterexample}
	Let $(G',h')$ be defined by the PEL-datum
		\[ (\q(i), *, \q(i)^{\op 2}, (-\tr_{\q(i)/\q}(xi\bar y)\op \tr_{\q(i)/\q}(xi\bar y)) ,h')   \]
	where $h'\colon \co \to \End_\re(\co^{\op 2})$ is the map which sends $z$ to multiplication by $(z,\bar z)$. We write $\GU_{1,1}$ for $G'$. Then $(\GU_{1,1})_\re$ coincides with the usual generalised unitary group of complex matrices preserving, up to scaling, a Hermitian form of signature $(1,-1)$.
	
	Let $\chi$ denote the two dimensional representation of $\GU_{1,1}$ given by the composition
	\[\chi \colon  \GU_{1,1}\overset{\det}{\longrightarrow} \Res_{\q(i)/\q} \G_m \overset{z/\bar z}{\longrightarrow} U_1 \longrightarrow \GL(\q(i)).   \]
	Here, the determinant is given by considering $\GU_{1,1}\subset \Aut_{\q(i)}(\q(i)^{\op 2})$ whilst $U_1$ denotes the norm one elements of $\q(i)$ and the final map is given by the action of $U_1$ on $\q(i)$ by multiplication. Note that the image of $\chi$ preserves the symmetric non-degenerate pairing $\tr_{\q(i)/\q}(a\bar b)$ and that, after base change to $\re$, $\chi$ is trivial on the image of $h'$.
	
	Now let $V'$ denote the standard representation of $\GU_{1,1}$ and consider the representation
	\[ \GU_{1,1} \longrightarrow \GSp(V')\ti \operatorname{GO}(\q(i)) \overset{\ot}{\longrightarrow} \GSp(V'\ot_\q \q(i)) . \]
	Since $\chi$ is trivial on $\im h'$, this is a morphism of Shimura data when $\GSp(V'\ot_\q \q(i))$ is given the PEL-datum 
	\[(\q, *=\id, V'\ot_\q \q(i),\la \ , \ \ra_{V'}\ot \q(i),h(z)=\left(\begin{array}{c c }
	z & 0 \\0& \bar z 
	\end{array}  \right) \ot \id ).\]	
	It can be checked that $f^*(V'\ot_\q \q(i)) \cong V' \ot \chi$ is not isomorphic to $V'^{\op 2}$, for example by base changing to $\co$ where $\GU_{1,1}$ becomes isomorphic to $\G_m\ti \GL_2$. As a result, $f$ is not admissible.
\end{example}
In contrast, in the symplectic case there are no non-admissible morphisms. In particular, there do not exist non-trivial representations $\chi$ which are trivial on the image of $h$ in the symplectic case.
\begin{lemma}\label{AVsummand}
	Let $(G,\mathfrak{X})$ be a Shimura datum with a choice of PEL-datum $(B,*,V,\la \ , \ \ra , h)$ for which $B_\re$ only has factors of symplectic type (in the sense of Definition \ref{Bdecomp}). Then all objects of $\Rep(G)^\AV$ are direct summands of $V^{\op k}$ for some $k$.
\end{lemma}
\begin{proof}
	It suffices to show the analogous statement after base change to $\co$. Let $W$ be a $\co$-representation of $G_\co$ of Hodge type $\{(-1,0),(0,-1)\}$. By Lemma \ref{PELsatisfySV5}, $G_{1,\co}\cong\prod_i \Sp_{2m_i}$. Accordingly, $W|_{G_{1,\co}}$ splits as a direct sum of irreducibles on which $G_{1,\co}$ acts via projection to some simple factor.
	\begin{claim}
		Let $T$ be an irreducible representation of $\Sp_{2n}$ that upon restriction to the subspace
		\[\S \supset U_1 \cong  \left\{ \left(\begin{array}{c | c}
		aI_g & -bI_g \\\hline bI_g & aI_g 
		\end{array}\right) \ \middle | \ a^2+b^2=1  \right\} \]
		$(z\op\bar z)$-isotypical. Then $T$ is isomorphic to the standard representation.
	\end{claim}
	\begin{proof}[Proof of Claim]
		Since $\Sp_{2n}$ is simply connected, it is equivalent to show the analogous statement for irreducible representations of $\mathfrak{sp}_{2n}$ that upon restriction to
		\[ \left(\begin{array}{c | c} 0& -bI_g \\\hline bI_g& 0 \end{array}  \right)   \]
		have weights $\{1,-1\}$. In the notation of \cite[Sec.\ 17.2]{FultonHarris}, these are precisely irreducible representations with highest weight $\lambda_1 L_1+...+ \lambda_n L_m$, for $\lambda_i$ integers with $\lambda_1\ge \lambda_2 \ge ...\ge \lambda_n \ge 0$ and for which $\sum \lambda_i =1$. As such, the only possible highest weight is $L_1$, which does indeed correspond to the standard representation.		
	\end{proof}
	From the proof of Lemma \ref{PELsatisfySV5}, the image of $\U_1$ under $h$ can be assumed to have the form given in the claim on each simple factor and since $W$ is of Hodge type $\{(-1,0),(0,-1)\}$, $W|_{h(\U_1)}$ is $(z\op\bar z)$-isotypical. By the claim, the irreducible factors of $W|_{G_{1,\co}}$ must be summands of the standard representations of the corresponding factor of $G_{1,\co}$, and therefore also of $V_\co$.
	
	Since the action of scalar matrices on $W$ is determined by its weight, the functor $\Rep(G_\co)^\AV \to \Rep(G_{1,\co})$ is faithful. In particular, there is at most one representation, up to isomorphism, of $G'_\co$ of Hodge type $\{(-1,0),(0,-1)\}$ restricting to any representation of $G_{1,\co}$. Since all irreducible representations of $W|_{G_{1,\co}}$ are summands of the standard representation of $G'_{1,\co}$ and the standard representation of $G_\co$ is one representation restricting to the standard representation of $G_{1,\co}$, we must have the all irreducible objects of $\Rep(G_\co)^\AV$ are direct summands of $V_\co$.
\end{proof}
Lemma \ref{AVsummand} does not hold in the orthogonal case, but is true upon restriction to $G^\circ$.

Under the assumptions of the lemma $\Anc_G$ extends $\mu_G^\mot$ up to natural isomorphism (see Remark \ref{Anconaextends}). We also find:
\begin{corollary}\label{symp/orth}
	Let $(G',\mathfrak{X}')$ be a Shimura datum with a choice of PEL-datum $(B',*',V',\la \ , \ \ra' , h')$ for which $B_\re$ only has factors of symplectic type (in the sense of Definition \ref{Bdecomp}). Then for any Shimura datum $(G,h)$ with a choice of PEL-datum, any map $f \colon (G',h')\to (G,h)$ is admissible (i.e.\ satisfies $(\star)$ of Definition \ref{admissibledef}).
\end{corollary}
\bibliographystyle{amsalphainitials}
\bibliography{MyCollection}

\providecommand{\bysame}{\leavevmode\hbox to3em{\hrulefill}\thinspace}
\providecommand{\MR}{\relax\ifhmode\unskip\space\fi MR }
\providecommand{\MRhref}[2]{%
  \href{http://www.ams.org/mathscinet-getitem?mr=#1}{#2}
}
\providecommand{\href}[2]{#2}
\begin{thebibliography}{Lem17}

\bibitem[AK02]{AndreKahn}
Y.~Andr\'e and B.~Kahn, \emph{Nilpotence, radicaux et structures mono\"\i
  dales}, Rend. Sem. Mat. Univ. Padova \textbf{108} (2002), 107--291, With an
  appendix by Peter O'Sullivan.

\bibitem[Anc15]{Anconapaper}
G.~Ancona, \emph{D\'ecomposition de motifs ab\'eliens}, Manuscripta Math.
  \textbf{146} (2015), no.~3-4, 307--328.

\bibitem[CH00]{CortiHanamura}
A.~Corti and M.~Hanamura, \emph{Motivic decomposition and intersection {C}how
  groups. {I}}, Duke Math. J. \textbf{103} (2000), no.~3, 459--522.

\bibitem[Del71]{travauxdeShimura}
P.~Deligne, \emph{Travaux de {S}himura}, 123--165. Lecture Notes in Math., Vol.
  244. \MR{0498581}

\bibitem[Del79]{Deligne}
\bysame, \emph{Valeurs de fonctions {$L$}\ et p\'eriodes d'int\'egrales},
  Automorphic forms, representations and {$L$}-functions ({P}roc. {S}ympos.
  {P}ure {M}ath., {O}regon {S}tate {U}niv., {C}orvallis, {O}re., 1977), {P}art
  2, Proc. Sympos. Pure Math., XXXIII, Amer. Math. Soc., Providence, R.I.,
  1979, With an appendix by N. Koblitz and A. Ogus, pp.~313--346.

\bibitem[DM82]{DeligneMilne}
P.~Deligne and J.~S. Milne, \emph{Tannakian categories}, Hodge cycles, motives,
  and {S}himura varieties, Lecture Notes in Mathematics, vol. 900,
  Springer-Verlag, Berlin-New York, 1982.

\bibitem[DM91]{DeningerMurre}
C.~Deninger and J.~Murre, \emph{Motivic decomposition of abelian schemes and
  the {F}ourier transform}, J. Reine Angew. Math. \textbf{422} (1991),
  201--219.

\bibitem[FH91]{FultonHarris}
W.~Fulton and J.~Harris, \emph{Representation theory}, Graduate Texts in
  Mathematics, vol. 129, Springer-Verlag, New York, 1991, A first course,
  Readings in Mathematics.

\bibitem[Kas86]{Kashiwara}
M.~Kashiwara, \emph{A study of variation of mixed {H}odge structure}, Publ.
  Res. Inst. Math. Sci. \textbf{22} (1986), no.~5, 991--1024.

\bibitem[Kin98]{Kings}
G.~Kings, \emph{Higher regulators, {H}ilbert modular surfaces, and special
  values of {$L$}-functions}, Duke Math. J. \textbf{92} (1998), no.~1, 61--127.

\bibitem[Kot92]{Kottwitz}
R.~E. Kottwitz, \emph{Points on some {S}himura varieties over finite fields},
  J. Amer. Math. Soc. \textbf{5} (1992), no.~2, 373--444.

\bibitem[Lem17]{Lemma}
F.~Lemma, \emph{On higher regulators of {S}iegel threefolds {II}: the
  connection to the special value}, Compos. Math. \textbf{153} (2017), no.~5,
  889--946.

\bibitem[LSZ17]{LSZ}
D.~{Loeffler}, C.~{Skinner}, and S.~L. {Zerbes}, \emph{{Euler systems for
  GSp(4)}}, preprint \texttt{arXiv:1706.00201} (2017).

\bibitem[Mil17]{MilneSV}
J.~S. Milne, \emph{Introduction to {S}himura varieties}, 2017, Available at
  www.jmilne.org/math/, p.~172.

\bibitem[O'S11]{O'Sullivan}
P.~O'Sullivan, \emph{Algebraic cycles on an abelian variety}, J. Reine Angew.
  Math. \textbf{654} (2011), 1--81.

\bibitem[Pin90]{Pink}
R.~Pink, \emph{Arithmetical compactification of mixed {S}himura varieties},
  Bonner Mathematische Schriften [Bonn Mathematical Publications], vol. 209,
  Universit\"at Bonn, Mathematisches Institut, Bonn, 1990, Dissertation,
  Rheinische Friedrich-Wilhelms-Universit\"at Bonn, Bonn, 1989.

\bibitem[Pin92]{PinkSheaves}
\bysame, \emph{On {$l$}-adic sheaves on {S}himura varieties and their higher
  direct images in the {B}aily-{B}orel compactification}, Math. Ann.
  \textbf{292} (1992), no.~2, 197--240.

\bibitem[Tor18]{thesis}
A.~Torzewski, \emph{Regulator constants of integral representations, together
  with relative motives over shimura varieties}, Ph.D. thesis, University of
  Warwick, 2018.

\bibitem[Wil97]{Wild}
J.~Wildeshaus, \emph{Realizations of polylogarithms}, Lecture Notes in
  Mathematics, vol. 1650, Springer-Verlag, Berlin, 1997.

\end{thebibliography}
\vspace{0.1cm}
\textsc{Department of Mathematics, University College London, Gower Street, London,\\ WC1E 6BT, UK}\\
\noindent \emph{E-mail address:}	\texttt{alex.torzewski@gmail.com}

\end{document}